%% file: paper.tex
\PassOptionsToPackage{unicode}{hyperref}
\PassOptionsToPackage{hyphens}{url}
\documentclass[
  11pt,
  letterpaper]{siamart251216}
\usepackage{xcolor}
\usepackage[margin=1in]{geometry}
\usepackage{amsmath,amssymb}
\usepackage{iftex}
\ifPDFTeX
  \usepackage[T1]{fontenc}
  \usepackage[utf8]{inputenc}
  \usepackage{textcomp} 
\else 
  \usepackage{unicode-math} 
  \defaultfontfeatures{Scale=MatchLowercase}
  \defaultfontfeatures[\rmfamily]{Ligatures=TeX,Scale=1}
\fi
\usepackage{lmodern}
\ifPDFTeX\else
\fi
\IfFileExists{upquote.sty}{\usepackage{upquote}}{}
\IfFileExists{microtype.sty}{
  \usepackage[]{microtype}
  \UseMicrotypeSet[protrusion]{basicmath} 
}{}
\makeatletter
\@ifundefined{KOMAClassName}{
  \IfFileExists{parskip.sty}{%
    \usepackage{parskip}
  }{
    \setlength{\parindent}{0pt}
    \setlength{\parskip}{6pt plus 2pt minus 1pt}}
}{
  \KOMAoptions{parskip=half}}
\makeatother
\makeatletter
\ifx\paragraph\undefined\else
  \let\oldparagraph\paragraph
  \renewcommand{\paragraph}{
    \@ifstar
      \xxxParagraphStar
      \xxxParagraphNoStar
  }
  \newcommand{\xxxParagraphStar}[1]{\oldparagraph*{#1}\mbox{}}
  \newcommand{\xxxParagraphNoStar}[1]{\oldparagraph{#1}\mbox{}}
\fi
\ifx\subparagraph\undefined\else
  \let\oldsubparagraph\subparagraph
  \renewcommand{\subparagraph}{
    \@ifstar
      \xxxSubParagraphStar
      \xxxSubParagraphNoStar
  }
  \newcommand{\xxxSubParagraphStar}[1]{\oldsubparagraph*{#1}\mbox{}}
  \newcommand{\xxxSubParagraphNoStar}[1]{\oldsubparagraph{#1}\mbox{}}
\fi
\makeatother

\usepackage{longtable,booktabs,array}
\usepackage{calc} 
\usepackage{etoolbox}
\makeatletter
\patchcmd\longtable{\par}{\if@noskipsec\mbox{}\fi\par}{}{}
\makeatother
\IfFileExists{footnotehyper.sty}{\usepackage{footnotehyper}}{\usepackage{footnote}}
\makesavenoteenv{longtable}
\usepackage{graphicx}
\makeatletter
\newsavebox\pandoc@box
\newcommand*\pandocbounded[1]{
  \sbox\pandoc@box{#1}%
  \Gscale@div\@tempa{\textheight}{\dimexpr\ht\pandoc@box+\dp\pandoc@box\relax}%
  \Gscale@div\@tempb{\linewidth}{\wd\pandoc@box}%
  \ifdim\@tempb\p@<\@tempa\p@\let\@tempa\@tempb\fi
  \ifdim\@tempa\p@<\p@\scalebox{\@tempa}{\usebox\pandoc@box}%
  \else\usebox{\pandoc@box}%
  \fi%
}
\def\fps@figure{htbp}
\makeatother

\usepackage[numbers]{natbib}
\usepackage{amsmath}
\usepackage{amssymb}
\usepackage{mathtools}

\usepackage{booktabs}
\usepackage{array}
\usepackage{makecell}
\usepackage{caption}
\usepackage{tikz}
\usetikzlibrary{shapes.geometric, arrows}

\usepackage[ruled,vlined,linesnumbered,algo2e]{algorithm2e}

\usepackage{acro}
\acsetup{
  make-links = false,
  pages/display = first
}
\input{Abbreviations}

\makeatletter
\@ifpackageloaded{caption}{}{\usepackage{caption}}
\AtBeginDocument{%
\ifdefined\contentsname
  \renewcommand*\contentsname{Table of contents}
\else
  \newcommand\contentsname{Table of contents}
\fi
\ifdefined\listfigurename
  \renewcommand*\listfigurename{List of Figures}
\else
  \newcommand\listfigurename{List of Figures}
\fi
\ifdefined\listtablename
  \renewcommand*\listtablename{List of Tables}
\else
  \newcommand\listtablename{List of Tables}
\fi
\ifdefined\figurename
  \renewcommand*\figurename{Figure}
\else
  \newcommand\figurename{Figure}
\fi
\ifdefined\tablename
  \renewcommand*\tablename{Table}
\else
  \newcommand\tablename{Table}
\fi
}
\@ifpackageloaded{float}{}{\usepackage{float}}
\floatstyle{ruled}
\@ifundefined{c@chapter}{\newfloat{codelisting}{h}{lop}}{\newfloat{codelisting}{h}{lop}[chapter]}
\floatname{codelisting}{Listing}

\makeatother
\makeatletter
\@ifpackageloaded{caption}{}{\usepackage{caption}}
\@ifpackageloaded{subcaption}{}{\usepackage{subcaption}}
\makeatother
\usepackage{bookmark}
\IfFileExists{xurl.sty}{\usepackage{xurl}}{} 
\hypersetup{
  pdftitle={Spectral Effects of Heavy-Tailed Vertex Noise in Geometric Graphs},
  pdfauthor={Ben Cardoen; Jeremy Budd; Enrico Amico; Ghassan Hamarneh; Fabian Spill},
  pdfkeywords={spectral graph theory, vertex noise, separability},
  hidelinks,
  pdfcreator={LaTeX via pandoc}}

\title{Spectral Effects of Heavy-Tailed Vertex Noise in Geometric
Graphs}
\author{Ben Cardoen \and Jeremy Budd \and Enrico Amico \and Ghassan
Hamarneh \and Fabian Spill}
\date{2026-04-16}
\headers{Heavy-Tailed Vertex Noise in Geometric Graphs}{Cardoen, Budd, Amico, Hamarneh, and Spill}

\begin{document}
\maketitle
\begin{abstract}
We characterize which local matrix structures saturate Weyl's eigenvalue
perturbation bound for graph Laplacians under geometrically constrained
vertex displacements. Geometric graphs with heavy-tailed vertex noise
arise across sensor networks, biological imaging, and spatial omics, yet
tractable predictions for noise-induced spectral error remain limited.
While Weyl's inequality guarantees eigenvalue continuity, it does not
explain how local geometric perturbations propagate to global spectral
error, nor how tight these bounds are in structured geometric settings.
We study geometric graphs abstracted from biophysical systems,
incorporating clearance, planarity, and identifiability constraints that
govern physically realizable embeddings. Within this constrained
setting, we identify witness motifs, small subgraphs in maximally
noise-sensitive geometric configurations, that dominate weighted-degree
and graph Laplacian spectral perturbations under tempered power-law
vertex displacements. This motif decomposition reduces global spectral
sensitivity to a finite catalog of local extremal structures and
identifies configurations that attain Weyl-tight bounds. We then lift
these constrained-graph results to general straight-line embedded graphs
in arbitrary dimension via local repair operations producing a
constrained surrogate graph that preserves sensitivity-relevant
structure. To quantify noise-induced spectral variation in both
strong-oracle (ground truth known) and weak-oracle (observations only)
regimes, we introduce stochastic co-spectrality (SC) and the stochastic
spectral separation index (S3I), which characterize when observed
spectral distances are noise-driven and when noise parameters are
separable. Together, these results provide a principled pathway from
local geometric noise to global spectral error in graph Laplacian
matrices, enabling estimation of spectral fragility from graph structure
without exhaustive eigenvalue computation or restrictive distributional
assumptions beyond moment bounds.
\end{abstract}

\begin{keywords}
spectral graph theory, vertex noise, geometric graphs, heavy-tailed perturbations, graph Laplacian, separability
\end{keywords}

\begin{MSCcodes}
05C50, 15A18, 60E07, 62R40
\end{MSCcodes}

\section{Introduction}\label{sec-intro}

Weyl's inequality bounds eigenvalue displacement of a Hermitian matrix
under additive perturbation, but does not identify which perturbation
structures attain the bound. For graph Laplacians under geometrically
constrained vertex noise, we show that a finite catalogue of small
subgraph configurations (witness motifs) saturate or nearly saturate
this bound (Table~\ref{tbl-degree-summary}), reducing global spectral
sensitivity analysis to local extremal geometry. Concretely, in many
physical and biological domains, graphs are not static combinatorial
objects but geometric abstractions of spatially embedded systems,
constructed from noisy measurements of vertex positions. In
super-resolution microscopy, for example, structures are inferred from
fluorophore localisations subject to blinking, drift, and finite
localisation precision \citep{Rust2006}. In imaging-based spatial
transcriptomics such as MERFISH, molecules are similarly detected and
localised in situ under microscopy noise and occasional outlier errors
\citep{Chen2015MERFISH}. In sensor networks, node positions are
estimated from imperfect localisation procedures and can exhibit
heterogeneous errors and outliers. In such contexts, spectral invariants
are frequently employed to compare network topologies or classify
structural phenotypes. Because edge weights depend nonlinearly on
inter-vertex distance, vertex-position noise induces correlated,
non-Gaussian perturbations of edge weights and Laplacians that are not
captured by independent edge-noise models. This geometric noise can
compress or inflate spectral distances in a way governed by local
geometric configurations, motivating structure-aware predictions of
spectral error.

While Gaussian models capture small localisation uncertainty, extremal
spectral perturbations in geometric graphs are driven by rare moderate
displacements. Pure power-law models overemphasise these events and
yield infinite-variance regimes incompatible with spectral control. We
therefore adopt a \ac{CTPL} displacement model, that exhibits a
power-law regime for moderate jumps with exponential tempering at large
radii, ensuring finite moments. Our analysis is application-agnostic and
targets the structural mechanisms by which geometric noise propagates
from local configurations to Laplacian spectra. We work in a weak-oracle
setting in which only perturbed vertex positions are observed
(Figure~\ref{fig-graphmodel}), rendering all edge weights random even
when the underlying graph is fixed. Within this setting, we study both
spectral fragility---eigenvalue displacement under noise---and spectral
distinguishability---the ability to separate graphs or noise regimes
using spectral distances. Although the \ac{CTPL} model is defined for
straight-line embeddings in \(\mathbb{R}^d\), all detailed geometric
arguments, motif catalogues, and numerical validations in this paper are
instantiated in two dimensions. This planar setting captures common
spatial networks and enforces a finite, explicitly characterisable
family of witness motifs via packing constraints; higher-dimensional
extensions are discussed only at the level of general formulation.

Classical matrix perturbation theory provides worst-case bounds on
eigenvalue and eigenvector changes under additive perturbations
\citep{StewartSun1990, Bhatia1997}, but these results are largely
distribution-free and do not encode the geometry of vertex noise.
Spectral convergence of graph Laplacians has been rigorously analysed
for random geometric and point-cloud graphs under static, light-tailed
vertex sampling
\citep{Penrose2003Random, Hein2007GraphLaplacians, GarciaTrillosFoCM2020},
yet these frameworks do not address heavy-tailed geometric perturbations
of a fixed underlying graph or link local geometry to global spectral
bounds. Related work on network robustness focuses on percolation, edge
failures, and degree fluctuations
\citep{Barabasi1999Emergence, Holme2017}, but typically treats vertex
positions as fixed or light-tailed latent variables, as in classical
random geometric graph models \citep{Penrose2003Random}.

In this paper, we:

\begin{enumerate}
\def\labelenumi{\alph{enumi}.}
\item
  introduce a geometric vertex-noise model based on \ac{TT} radial
  displacements, implemented via a \ac{CTPL} distribution, and show how
  it induces correlated, non-Gaussian edge-weight perturbations even
  when the underlying graph is fixed.
\item
  derive closed-form and asymptotic bounds for spectral displacement by
  reducing global perturbations to a finite catalogue of extremal
  witness motifs---small, degree-bounded subgraphs whose geometry
  certifies worst-case eigenvalue shifts.
\item
  formalise stochastic co-spectrality and the stochastic spectral
  separation index as observable summaries that quantify when spectral
  distances are noise-dominated and when noise parameters cease to be
  statistically separable.
\item
  prove that straight-line embedded geometric graphs admit bounded-error
  reductions via local repair operations, and establish a
  degree-controlled Frobenius perturbation bound that lifts planar
  motif-based guarantees to unconstrained embeddings.
\end{enumerate}

\section{Methods}\label{sec-methods}

A summary of the principal notation introduced in this section is
collected in Table~\ref{tbl-notation} (Appendix).

\subsection{Definitions, notations, and preliminaries}\label{sec-defs}

A weighted undirected graph with a straight-line embedding in
\(\mathbb{R}^d\) is a triple \(G = (V, E, W)\), where
\(V = \{v_1,\ldots,v_n\}\) is a finite vertex set,
\(E \subseteq \{\{v_i,v_j\} \mid i \neq j\}\) is a set of undirected
edges, and \(W : E \to \mathbb{R}_{>0}\) assigns positive weights to
edges. Each vertex \(v_i\) is associated with a coordinate
\(x_i \in \mathbb{R}^d\), inducing a straight-line embedding of \(G\).
The weighted adjacency matrix \(A \in \mathbb{R}^{n \times n}\) is
defined by \(A_{ij} = W(\{v_i,v_j\})\) if \(\{v_i,v_j\} \in E\) and
\(A_{ij} = 0\) otherwise, and the weighted degree matrix \(D\) is
diagonal with entries \(D_{ii} = \sum_j A_{ij}\). We use the
combinatorial Laplacian \(L = D - A\), whose eigenvalues
\(0 = \lambda_1 \le \lambda_2 \le \dots \le \lambda_n\) are real. To
quantify spectral differences between graphs, we use the one-dimensional
Wasserstein--2 distance. Given eigenvalue sets
\(\boldsymbol{\lambda}=\{\lambda_1,\ldots,\lambda_n\}\) and
\(\boldsymbol{\lambda}'=\{\lambda'_1,\ldots,\lambda'_m\}\), define the
empirical spectral measures
\(\mu_G = \frac{1}{n}\sum_{i=1}^n \delta_{\lambda_i}\) and
\(\mu_{G'} = \frac{1}{m}\sum_{j=1}^m \delta_{\lambda'_j}\). In one
dimension, the squared Wasserstein--2 distance admits the explicit form
\(W_2^2(\mu_G,\mu_{G'}) = \int_0^1 \bigl(F_G^{-1}(t)-F_{G'}^{-1}(t)\bigr)^2\,dt\),
where \(F_G^{-1}\) and \(F_{G'}^{-1}\) denote the quantile
functions\citep{santambrogio2015optimal}; equivalently,
\(W_2^2(\mu_G,\mu_{G'}) = \sum_{\ell=1}^L \omega_\ell (x_\ell-y_\ell)^2\)
with weights \(\omega_\ell\) summing to one and \(x_\ell,y_\ell\) drawn
from the sorted spectra. The spectral distance used throughout the paper
is the normalized Wasserstein distance
\begin{equation}\phantomsection\label{eq-w2-practical}{
d_{\mathrm{SD}}(G,G')
= \frac{1}{\lambda_{\mathrm{ref}}} W_2(\mu_G,\mu_{G'}), 
\qquad
\lambda_{\mathrm{ref}}
= \max\bigl(\|\boldsymbol{\lambda}\|_\infty,\|\boldsymbol{\lambda}'\|_\infty\bigr).
}\end{equation} Because \(\lambda_{\mathrm{ref}}\) depends on both the
unperturbed and perturbed spectra, it can be affected by noise where
that noise alters the maximal weighted degree. No single normalization
scheme is universally optimal: alternatives such as normalizing by the
unperturbed \(\|\boldsymbol{\lambda}\|_\infty\) alone avoid stochastic
denominators but sacrifice symmetry, while fixed scaling constants lose
adaptivity to graph size. The choice of \(\lambda_{\mathrm{ref}}\) does
not affect the theoretical bounds or proofs, which operate on
unnormalized spectral shifts; it enters only in the empirical metrics
(SC, S3I). The accompanying code permits users to substitute
domain-specific normalization as needed.

\subsection{Tempered Lévy Distribution}\label{sec-tl-distribution}

To study heavy tail noise, we need a carefully selected class of
distributions. The classical Lévy distribution is a special case of the
stable distribution family with stability parameter \(\alpha = 1/2\) and
skewness parameter \(\beta = 1\). Stable distributions with
\(\alpha < 2\) exhibit power-law tails and are often called Lévy
\(\alpha\)-stable distributions. In this work, we model vertex
perturbations using an \ac{CTPL} distribution, which modifies a
power-law tail by exponential decay (Figure~\ref{fig-tl}, Panel A). This
distribution belongs to the class of tempered stable
distributions\citep{Rosinski2007} but avoids the sequential-process
connotation of ``Lévy process.''

\subsubsection{Base Lévy--Stable Power-Law Kernel
(1D)}\label{sec-levy-base}

We model jump magnitudes using the standard one-dimensional Lévy
distribution with location \(\mu\) and scale \(c>0\). Its probability
density function on \((\mu,\infty)\) is \[
f_{\text{Levy}}(x; \mu, c)
= \sqrt{\frac{c}{2\pi}}\,(x-\mu)^{-3/2}
  \exp\!\left(-\frac{c}{2(x-\mu)}\right),
\qquad x>\mu.
\] This is a proper probability density with heavy power-law-like tails
and infinite variance in its untempered form. In what follows, this
Lévy--stable law serves solely as a canonical power-law kernel: we use
it to define a single-step radial displacement model and do not invoke
any Lévy flight or process structure.

\subsubsection{Calibrated Exponentially Tempered Power Law
(CTPL)}\label{sec-ctpl}

To ensure the mathematical tractability of spectral observables and the
existence of moments, we apply exponential tempering to the Lévy
density. Exponentially tempered and truncated Lévy laws are well
established as a way to retain Lévy-like jump behaviour while enforcing
finite moments and tail control (e.g.
\citep{Koponen1995, Rosinski2007, ContTankov2004}). To avoid any
ambiguity in terminology and interpretation, although the untempered
kernel coincides with the \(\alpha=1/2\) Lévy--stable law, we use CTPL
strictly as a single-step radial displacement model and do not assume
any Lévy flight or process structure. For parameters \(c>0\),
\(\lambda\ge 0\), and location \(\mu\), define the tempered power-law
pdf by \[
\tilde{f}_{\text{temp}}(x; c, \lambda, \mu)
= f_{\text{Levy}}(x; \mu, c)\, e^{-\lambda x}, \qquad x>\mu,
\] and let \[
Z(c,\lambda,\mu)
= \int_{\mu}^{\infty} f_{\text{Levy}}(x; \mu, c)\, e^{-\lambda x}\,dx
< \infty
\] denote the normalising constant. The corresponding calibrated
tempered Lévy (CTPL) density is \[
f_{\text{CTPL}}(x; c, \lambda, \mu)
= \frac{1}{Z(c,\lambda,\mu)}\, f_{\text{Levy}}(x; \mu, c)\, e^{-\lambda x},
\qquad x>\mu.
\] Exponential tempering ensures the existence of all moments by
enforcing rapid decay at infinity, while the Lévy kernel's exponential
factor controls the singularity as \(x \to \mu^+\); together these imply
that \(Z(c,\lambda,\mu) < \infty\) and that all moments of
\(f_{\text{CTPL}}\) exist, while the Lévy core retains a heavy-tail-like
regime for moderate jump sizes. The parameter \(c\) controls the
power-law regime, while \(\lambda\) controls the exponential tempering.
The 2D isotropic perturbation model is
\(\boldsymbol{\xi} = R \cdot (\cos \Theta, \sin \Theta)\), where
\(R \sim \text{CTPL}(c, \lambda)\) and
\(\Theta \sim \text{Uniform}(0, 2\pi)\) independently. For
\(r \ll 1/\lambda\), the \ac{CTPL} density retains a power-law core with
exponent \(-3/2\) (up to multiplicative constants), so large
displacements are substantially more probable than under light-tailed
models such as Gaussian noise. For \(r \gg 1/\lambda\), exponential
tempering dominates and \(f_{\ac{CTPL}}(r) \sim e^{-\lambda r}\),
ensuring a well-defined tempered regime. As a consequence of exponential
tempering, all moments of the \ac{CTPL} law exist, guaranteeing that
spectral observables such as \ac{SC} and \ac{S3I} are well defined. In
the calibrated pre-tempering regime, this power-law core yields higher
tail probabilities than Gaussian or Gamma alternatives with exponential
decay (Figure~\ref{fig-tl}, Panel B), inducing heavy-tailed edge-weight
noise that cannot be captured by exponentially bounded models. In
applications, extremely large jumps are biophysically implausible and
also problematic for identifiability. We therefore work with a
calibrated and clipped version of the CTPL law. Given a maximum
admissible jump radius \(r_{\max}>0\) and a tail tolerance
\(\delta\in(0,1)\), we numerically tune the tempering parameter
\(\lambda=\lambda(c,r_{\max},\delta)\) so that the untruncated \ac{CTPL}
law satisfies \(\Pr[X>r_{\max}] \le \delta\) for
\(X\sim\mathrm{CTPL}(c,\lambda,\mu)\) (Figure~\ref{fig-tl}, Panel C).
Sampling is performed from the corresponding tempered power-law density
and deterministically clipped at \(r_{\max}\) (Algorithms
\ref{alg:sample-tl} and \ref{alg:autotune-lambda}), so that the law on
\([0,r_{\max})\) is unchanged and clipping concentrates at most
\(\delta\) of the total probability mass into an atom at \(r_{\max}\);
since \(\lambda\) is tuned to enforce \(\Pr[X > r_{\max}] \le \delta\),
this atom is controlled by the calibration tolerance. All analytic
arguments in this paper are restricted to displacements of magnitude at
most \(r_{\max}\). Although \(\lambda\) is calibrated only to control
the extreme tail beyond \(r_{\max}\), varying the Lévy scale parameter
\(c\) induces genuinely different behaviour in the core and
moderate-tail regimes, so the calibrated CTPL family does not collapse
to a single effective noise model.

\begin{figure}[htbp]

\centering{

\pandocbounded{\includegraphics[keepaspectratio]{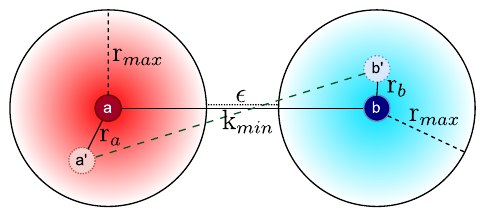}}

}

\caption{\label{fig-graphmodel}Schematic illustration of the vertex
noise model (see Section~\ref{sec-graphmodel} for full assumptions
A1--A4). Two vertices \(a\) and \(b\) have latent (noise-free) positions
separated by the true edge length \(k\) (strong oracle, solid line),
where \(k\) denotes the ground-truth length of a given edge and
satisfies \(k \ge k_{\min}=2r_{\max}+\varepsilon\); the lower bound
\(k_{\min}\) is a model constraint ensuring that perturbation balls
remain \(\varepsilon\)-separated so that vertices stay identifiable
under noise. Each vertex is independently displaced by a radial jump
\(r_a\), \(r_b\) drawn from a \ac{CTPL} distribution with uniformly
sampled direction, clipped at \(r_{\max}\); the perturbed coordinate
model is given in Section~\ref{sec-ctpl}. The observed edge length (weak
oracle, dashed green line) is computed from the perturbed positions
\(a'\), \(b'\).}

\end{figure}%

\begin{figure}[htbp]

\centering{

\includegraphics[width=\textwidth,height=\textheight,keepaspectratio]{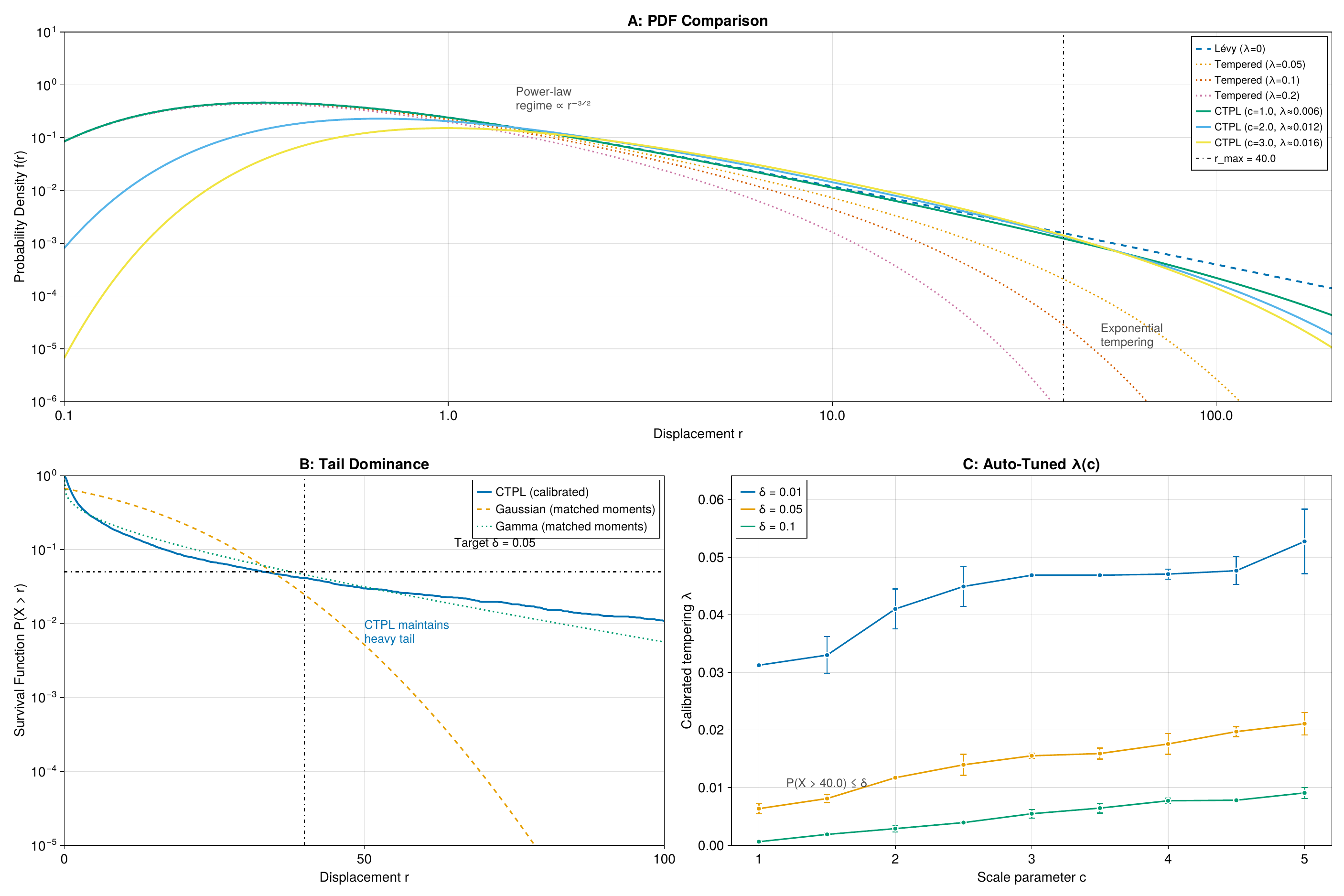}

}

\caption{\label{fig-tl}From Lévy noise to calibrated tempered power-law
(CTPL) noise. \textbf{(A)} Probability density functions showing the
effect of exponential tempering \(\lambda\) on a Lévy distribution,
preserving a power-law regime while suppressing extreme displacements
beyond a clipping scale. \textbf{(B)} Survival functions demonstrate
that the calibrated CTPL distribution retains tail dominance relative to
moment-matched Gaussian (\(\mu_G = \hat{E}[X_{\mathrm{CTPL}}]\),
\(\sigma^2 = \widehat{\mathrm{Var}}[X_{\mathrm{CTPL}}]\)) and Gamma
(\(\alpha = \hat{E}[X]^2 / \widehat{\mathrm{Var}}[X]\),
\(\theta = \widehat{\mathrm{Var}}[X] / \hat{E}[X]\)) alternatives under
a fixed tail probability constraint. \textbf{(C)} Empirical auto-tuning
of \(\lambda\) as a function of scale parameter \(c\), obtained via
rejection sampling to enforce \(P(X > r_{\max}) \le \delta\).}

\end{figure}%

We distinguish a strong oracle, which has access to the noise-free
vertex positions and true edge weights, from a weak oracle, which
observes only perturbed vertex positions and hence noisy edge weights
(Figure~\ref{fig-graphmodel}).

\subsection{Geometric Vertex-Noise Graph Model}\label{sec-graphmodel}

We consider straight-line embeddings \(x:V\to\mathbb{R}^d\) of graphs
\(G=(V,E)\) under the following geometric and noise assumptions.

(A1) \(\varepsilon\)-thick embedding (clearance). All non-incident
primitives are \(\varepsilon\)-separated; equivalently, the
\(\varepsilon/2\)-balls around vertices and \(\varepsilon/2\)-radius
tubes around edges are pairwise disjoint. Formally,
\begin{equation}\phantomsection\label{eq-thickness}{
\begin{aligned}
\min\{\,&\mathrm{dist}(x(u),x(v)),\ \mathrm{dist}(x(u),[x(a),x(b)]), \\
&\mathrm{dist}([x(a),x(b)],[x(c),x(d)])\} \ \ge \ \varepsilon,
\end{aligned}
}\end{equation} for all distinct \(u,v\in V\) and all edges
\(\{a,b\},\{c,d\}\in E\) with no shared endpoints. This generalizes the
thick embedding constraint studied in geometric graph theory
\citep{KolmogorovBarzdin1967, GromovGuth2012, BarrettEtAl2023} and
relates to Federer's notion of positive reach \citep{Federer1959} from
geometric measure theory.

(A2) Identifiability scale. Let \(k\) denote the ground-truth length of
a given edge (Figure~\ref{fig-graphmodel}), with
\(k \ge k_{\min} = 2r_{\max}+\varepsilon\); this lower bound ensures
that perturbation balls of radius \(r_{\max}\) remain
\(\varepsilon\)-separated, so that vertices stay identifiable under
noise. More generally, every pair of distinct vertices satisfies
\(\|x_u-x_v\|\ge k_{\min}\), whether or not \(\{u,v\}\in E\); this
guarantees that all vertices remain identifiable after perturbation.
Neighbours of a vertex lie on (or outside) the \(k\)-sphere and must be
pairwise \(\varepsilon\)-separated. Writing
\begin{equation}\phantomsection\label{eq-theta}{
\theta = 2\arcsin\left(\frac{\varepsilon}{2k}\right),
}\end{equation} the maximum degree is bounded by the spherical packing
number at angular separation \(\theta\) \citep{ConwaySloane1999}.

(A3) \ac{CTPL} vertex noise. Vertices are perturbed independently by an
isotropic \ac{CTPL} radial law as defined in Section~\ref{sec-ctpl}.
Each displacement has the form \(\boldsymbol{\eta} = R\,\mathbf{U}\)
(Figure~\ref{fig-graphmodel}), where
\(R \sim \text{\ac{CTPL}}(c,\lambda,\mu)\) and \(\mathbf{U}\) is
uniformly distributed on the unit sphere \(S^{d-1}\). The tempering
parameter \(\lambda=\lambda(c,r_{\max},\delta)\) is calibrated so that
\(\Pr[R>r_{\max}]\le\delta\), and all displacements are truncated to
satisfy \(R\le r_{\max}\).

(A4) Weight regularity. Edge weights are bounded and Lipschitz in the
inter-vertex distance \(\rho\): there exist \(w_{\max},L_w<\infty\) such
that \begin{equation}\phantomsection\label{eq-weight-lipschitz}{
0\le w(\rho)\le w_{\max}, \qquad |w(\rho+\Delta \rho)-w(\rho)|\le L_w\,|\Delta \rho|.
}\end{equation}

\emph{Remark.} The extremality results in Section~\ref{sec-degreeproof}
additionally require that \(\varphi\) is nondecreasing and convex; all
other results (SC, S3I, repair bounds) use only (A1)--(A4) as stated.

Under (A1)--(A3), per-edge changes \(|\delta w|\) are bounded, so
per-motif perturbations admit finite operator/Frobenius envelopes.

If the ambient metric has finite doubling dimension (every ball of
radius \(r\) is coverable by \(N_D\) balls of radius \(r/2\)), then
packing on the \(k\)-sphere obeys
\begin{equation}\phantomsection\label{eq-doubling-pack}{
\mathrm{Pack}\big(S^{d-1}(k),\varepsilon\big) \ \le\ C\left(\frac{k}{\varepsilon}\right)^{\log_2 N_D - 1},
}\end{equation} yielding a finite degree bound \citep{Heinonen2001}.

For the remainder of this paper we restrict to displacement metrics
\(d(x,y) = \|x-y\|\) induced by a norm \(\|\cdot\|\) on
\(\mathbb{R}^d\). Such metrics are convex and monotone nondecreasing in
displacement magnitude, the two properties required by the perturbation
and extremal arguments in Section~\ref{sec-degreeproof} and
Section~\ref{sec-motif-packing}. All \(\ell^p\)-norms (\(p \ge 1\)) are
admissible; non-Euclidean choices are angularly preferential and may
therefore induce norm-dependent extremal configurations.

\subsection{\texorpdfstring{Bounds on Spectral Change in \ac{CTPL}
Graphs}{Bounds on Spectral Change in  Graphs}}\label{sec-spectral-bounds}

The perturbation of vertices in a graph \(G\) leads to a perturbed
adjacency matrix \(A'\) and a corresponding change in the spectrum.
Characterizing this change is central to understanding the graph's
fragility. Under light-tailed noise (e.g.~Gaussian), extreme
displacements are exponentially unlikely; vertex perturbations can be
effectively concentrated within one standard deviation, reducing
correlated vertex noise to near-independent edge noise for which Weyl's
inequality yields tight, predictable bounds. Heavy-tailed \ac{CTPL}
noise changes this picture qualitatively. Because the extremal analysis
is parametric in the displacement radius, all bounds and motif
catalogues can be re-evaluated at any \(r \le r_{\max}\). Under
\ac{CTPL}, the power-law core assigns non-negligible probability mass to
displacement values well into the tail
range---e.g.~\(0.75\,r_{\max}\)---at probabilities far exceeding the
calibration cutoff \(\delta\). Re-computing the spectral bounds at these
sub-maximal but CTPL-probable displacements yields a continuum of
practically relevant perturbation estimates, not merely a single
worst-case number. Under Gaussian noise, this same range is effectively
empty, collapsing the analysis to the near-independent-edge regime.
Moreover, even moderate tail values suffice to break the
independent-edge-noise reduction, motivating the motif-level analysis
that follows. The spectral bounds themselves are deterministic: they
hold for any displacement bounded by \(r_{\max}\), regardless of the
generating distribution. In this section, we bound this change using two
complementary tools from matrix analysis: Cauchy interlacing is used to
derive deterministic lower bounds from locally perturbed induced
subgraphs \citep{Haemers1995, Chung1997Spectral}, while Weyl-type
perturbation bounds provide corresponding upper bounds on global
spectral variation \citep{Weyl1912, StewartSun1990, Bhatia1997}.

\subsubsection{Lower Bound on Perturbed Eigenvalues via Cauchy
Interlacing}\label{lower-bound-on-perturbed-eigenvalues-via-cauchy-interlacing}

We rely on the classical Cauchy interlacing theorem for eigenvalues of
principal submatrices \citep{Haemers1995, Chung1997Spectral}, which we
apply here to derive a lower bound on the spectral effect of localized
vertex perturbations.

\begin{theorem}[Induced-Subgraph Lower Bound for Perturbed Laplacians]\label{thm:induced-lower-bound}
Let $G'$ be a graph obtained from $G$ by vertex perturbations, and let $H' = G'[V_{H'}]$ be an induced subgraph on a vertex subset $V_{H'} \subseteq V(G')$.
Then the largest eigenvalue of the combinatorial Laplacian satisfies
$$
\lambda_n(L_{G'}) \;\ge\; \lambda_{\max}(L_{H'}).
$$
\end{theorem}

\begin{proof}
Let $L_{G'}[V_{H'},V_{H'}]$ denote the principal submatrix of $L_{G'}$ indexed by $V_{H'}$.
This matrix decomposes as
$$
L_{G'}[V_{H'},V_{H'}] = L_{H'} + \Delta_{H'},
\qquad
\Delta_{H'} := \operatorname{diag}\bigl(d_{G'}(v)-d_{H'}(v)\bigr)_{v\in V_{H'}} \succeq 0,
$$
where $\Delta_{H'}$ collects degree contributions from edges incident to $V_{H'}$ but not contained in $H'$.
By Cauchy interlacing, the largest eigenvalue of $L_{G'}$ satisfies
$$
\lambda_n(L_{G'}) \;\ge\; \mu_m\!\bigl(L_{G'}[V_{H'},V_{H'}]\bigr),
$$
where $\mu_m(\cdot)$ denotes the largest eigenvalue of the principal submatrix.
Since $L_{H'} = L_{G'}[V_{H'},V_{H'}] - \Delta_{H'}$ with $\Delta_{H'} \succeq 0$, monotonicity of eigenvalues under positive semidefinite perturbations yields
$$
\mu_m\!\bigl(L_{G'}[V_{H'},V_{H'}]\bigr) \;\ge\; \lambda_{\max}(L_{H'}).
$$
\end{proof}

This induced-subgraph bound can be conservative when the witness
subgraph omits high-degree connections to the remainder of the graph. In
such cases, complementary upper bounds based on operator-norm
perturbation theory, such as Weyl-type inequalities and
Hoffman--Wielandt bounds
\citep{Weyl1912, StewartSun1990, HoffmanWielandt1953}, may be tighter.

\subsubsection{Upper Bound on Eigenvalue Change via Weyl-Type
Perturbation
Bounds}\label{upper-bound-on-eigenvalue-change-via-weyl-type-perturbation-bounds}

We now derive a global upper bound on Laplacian eigenvalue displacement
under vertex perturbations. This bound follows directly from Weyl's
inequality for Hermitian matrix perturbations and complements the
induced-subgraph lower bounds of the previous section.

\begin{theorem}[Weyl-Type Upper Bound for Laplacian Perturbations]\label{thm:weyl-upper-bound}
Let $G$ be a graph with combinatorial Laplacian $L_G$, and let $G'$ be obtained from $G$ by vertex perturbations, inducing a Laplacian perturbation
$$
E := L_{G'} - L_G .
$$
Then, for each $k=1,\dots,n$,
\begin{equation}\label{eq:weyl-bound}
\bigl|\lambda_k(L_{G'}) - \lambda_k(L_G)\bigr|
\;\le\; \|E\|_2
\;\le\; \|E\|_F .
\end{equation}
\end{theorem}

\begin{proof}
Since $L_G$ and $L_{G'}$ are symmetric, Weyl's inequality for Hermitian matrices \cite{Weyl1912,StewartSun1990,Bhatia1997} implies that
$$
\bigl|\lambda_k(L_{G'}) - \lambda_k(L_G)\bigr| \le \|E\|_2
$$
for all $k$.
The second inequality follows from the standard bound $\|E\|_2 \le \|E\|_F$.
\end{proof}

This bound is global and sharp in principle, but it does not identify
which local perturbation patterns realise worst-case spectral
displacement. In particular, Weyl-type bounds obscure the role of
localized geometric configurations, motivating the motif-based and
subgraph-level analyses developed in the following sections.

\subsection{\texorpdfstring{The Correspondence between \ac{CTPL} Vertex
and Edge Noise
Distributions}{The Correspondence between  Vertex and Edge Noise Distributions}}\label{sec-tailweights}

Most existing models of noise on graphs act directly on edge weights,
typically treating perturbations as independent or weakly dependent
across edges, particularly in graph similarity and dynamic network
settings \citep{donnat2018tracking, Koutra2016DeltaCon}. If vertex-level
\ac{CTPL} noise reduced to a simple, tractable edge-noise law, spectral
fragility under vertex perturbations could be analysed directly within
these frameworks. In general, however, vertex noise induces edge-weight
displacements that are nonlinear, correlated, and geometry-dependent.
Understanding this distinction is essential, as it motivates the local,
motif-based analysis developed in the remainder of this section.
Consider a single edge \(\{a,b\}\) embedded in \(\mathbb{R}^2\) with
\(\|a-b\| = k\), and let \(a', b'\) be obtained by independent vertex
perturbations as in Assumption (A3). Write \(\eta_a = R_a U_a\) and
\(\eta_b = R_b U_b\), where \(R_a,R_b \sim \text{CTPL}(c,\lambda,\mu)\)
are independent radial jump magnitudes and \(U_a,U_b\) are independent
unit vectors with uniform directions. The noise-perturbed edge length is
\(k' = \|a' - b'\| = \| (a + \eta_a) - (b + \eta_b) \|,\) and the
induced edge-length perturbation is \(E = k' - k\). In polar
coordinates, \(E\) is a nonlinear function of the triple
\((R_a,R_b,\phi)\), where \(\phi\) denotes the relative angle between
the two perturbations. As a consequence, the induced edge-noise
distribution arises as a nontrivial pushforward of the joint
vertex-noise law and reflects geometric cancellation and reinforcement
effects. Even in this minimal setting, the resulting edge perturbation
is neither independent nor expressible in any simple, vertex-agnostic
parametric edge-noise model. For multi-edge configurations
(Section~\ref{sec-vn-noise-results}), shared vertices introduce strong
correlations between adjacent edge perturbations. These dependencies
persist under \ac{CTPL} vertex noise and are clearly visible
empirically, even in small geometric motifs
(Figure~\ref{fig-vertex-noise-is-not-edge-noise}). As a consequence,
vertex noise cannot, in general, be faithfully approximated by
independent or identically distributed edge-noise models. This
observation renders motif-level analysis essential: local configurations
provide the correct scale at which vertex-induced perturbations
concentrate and at which worst-case spectral effects can be identified.

\subsection{Witness Motifs for Spectral
Perturbation}\label{sec-degreeproof}

We isolate a bounded family of geometrically admissible subgraphs whose
local sensitivity under vertex perturbations yields sharp lower bounds
on global spectral change.

\begin{figure}[htbp]

\centering{

\includegraphics[width=4in,height=\textheight,keepaspectratio]{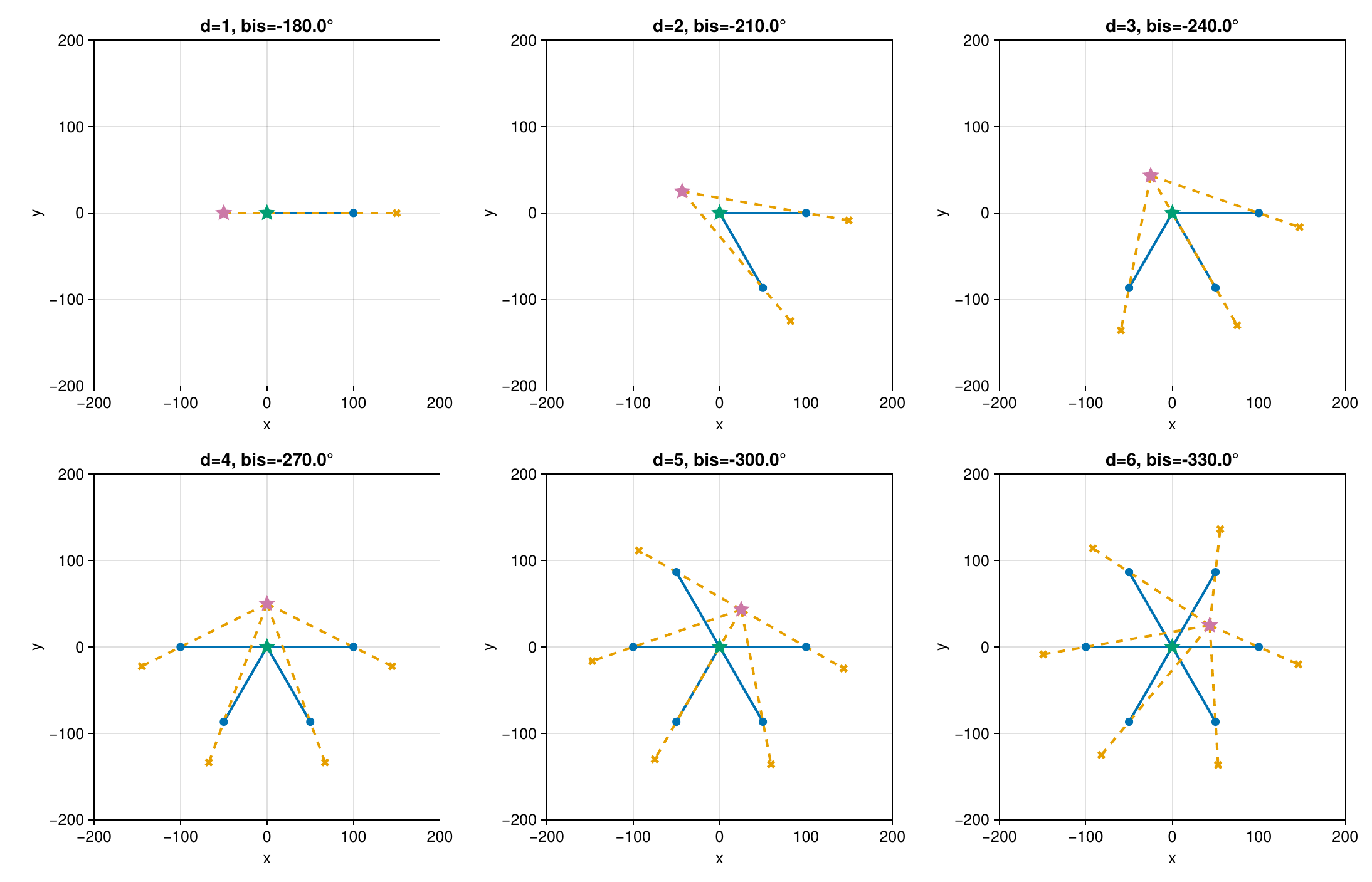}

}

\caption{\label{fig-motifs-worst-case}Canonical geometric motifs under
vertex noise. The six witness motifs considered in this work, shown in
two dimensions with a central hub and increasing degree \(n=1\)--\(6\).
Ground-truth positions are shown in blue (solid edges, filled circles
for spokes, green star for hub); perturbed positions in orange (dashed
edges, crosses for spokes, purple star for hub). For each motif, the
configuration shown corresponds to a maximally sensitive geometric
arrangement with respect to vertex noise, aligned along the bisector
direction indicated in the panel titles; sensitivity away from this
configuration is quantified in Figure~\ref{fig-motifs-optimality}.}

\end{figure}%

\subsubsection{Geometric Packing Bounds on Maximally Sensitive Hub
Motifs}\label{geometric-packing-bounds-on-maximally-sensitive-hub-motifs}

We quantify local sensitivity via the relative change in weighted degree
at a vertex \(v\) under geometric perturbation. Let \[
\mathrm{WD}_G(v) := \sum_{u \in N_G(v)} \varphi(\|x_v-x_u\|),
\] where \(\varphi\) is a differentiable, nondecreasing radial kernel on
the relevant distance range, and define for the perturbed graph \(G'\)
\begin{equation}\phantomsection\label{eq-reldegree}{
\delta_{\mathrm{rel}}\mathrm{WD}(v)
:= \frac{\mathrm{WD}_{G'}(v)-\mathrm{WD}_{G}(v)}{\mathrm{WD}_{G}(v)},
\qquad
\mathrm{WD}_{G}(v) > 0.
}\end{equation}

Under (A1)--(A4), extremal local configurations for
\(\delta_{\mathrm{rel}}\mathrm{WD}(v)\) arise when many incident edges
are simultaneously maximally sensitive to radial displacement while
remaining mutually \(\varepsilon\)-separated. The following two theorems
characterise this extremal problem in Euclidean space
(\(\|\cdot\| = \|\cdot\|_2\)): Theorem \ref{thm:packing-bound} bounds
the number of simultaneously sensitive neighbours via a packing argument
on \(\mathbb{S}^{d-1}\), and Theorem \ref{thm:hub-spoke-extremal}
locates the maximiser of the weighted-degree displacement within the
feasible region. Analogous results hold for other admissible norms, but
the packing geometry of the unit ball---and hence the extremal
catalogue---is norm-dependent.

\begin{theorem}[Packing Bound for Maximally Sensitive Hub Motifs]\label{thm:packing-bound}
Let $x_0 \in \mathbb{R}^d$ be a hub vertex with neighbors $x_1,\dots,x_n$, all at distance $\|x_i-x_0\|=k_{\min}$, under the Euclidean norm and assumptions (A1)--(A4).
Then $n \le \tau_d$, where $\tau_d$ is the kissing number in dimension $d$.
\end{theorem}

\begin{proof}
By (A2), every pair of distinct vertices satisfies $\|x_i-x_j\|\ge k_{\min}$, whether or not they share an edge.
Since all neighbors lie at distance $k_{\min}$ from the hub, the law of cosines gives, for $i\neq j$,
\[
  \|x_i-x_j\|^2 = 2k_{\min}^2(1-\cos\theta_{ij}),
\]
where $\theta_{ij}$ is the angle subtended at $x_0$.
The constraint $\|x_i-x_j\|\ge k_{\min}$ yields $\cos\theta_{ij}\le\tfrac{1}{2}$, i.e.\ $\theta_{ij}\ge\pi/3$.
The directions $u_i=(x_i-x_0)/k_{\min}\in\mathbb{S}^{d-1}$ therefore form a spherical code with minimum angular separation $\pi/3$, whose maximum cardinality is the kissing number $\tau_d$.
\end{proof}

In particular, for \(d=2\) we have \(\tau_2 = 6\), so the largest
geometrically admissible hub--spoke witness has degree six: the hub and
its neighbors form equilateral triangles, the tightest configuration
permitted by (A2). The extremal regime is \(\varepsilon\to 0\),
\(k=k_{\min}\to 2r_{\max}\), where the noise-to-edge-length ratio---and
hence \(\delta_{\mathrm{rel}}\mathrm{WD}\)---is maximised; for
\(k>k_{\min}\) the relative perturbation diminishes. The finite
catalogue (\(n \le 6\)) therefore corresponds to the maximally sensitive
regime \(k = k_{\min}\). The convex boundary argument underlying the
extremality theorem below holds in arbitrary dimension \(d\); the
restriction to \(d=2\) is made only to obtain the explicit witness
catalogue via \(\tau_2=6\). In higher dimensions, the same construction
applies with \(\tau_d\) replacing \(\tau_2\), yielding a larger
catalogue whose enumeration we do not pursue here.

\subsubsection{Geometric Extremality of
Motifs}\label{geometric-extremality-of-motifs}

\begin{theorem}[Geometric extremality of hub--spoke weighted-degree displacement]\label{thm:hub-spoke-extremal}
Let $d=2$.  Let $x_0 \in \mathbb{R}^2$ be a hub and $x_1,\dots,x_n$ its neighbors with $1 \le n \le \tau_2 = 6$ (Theorem~\ref{thm:packing-bound}), all at Euclidean distance $\|x_i-x_0\|=k_{\min}$ and satisfying $\angle x_i x_0 x_j \ge \pi/3$ for $i\neq j$.
Let edge weights be $w_i=\varphi(\|x_i-x_0\|)$, where $\varphi:[0,\infty)\to\mathbb{R}$ is differentiable, nondecreasing, and convex.
Suppose the hub may move within $B_0=\{h:\|h-x_0\|\le r_{\max}\}$ and each spoke endpoint within
$B_i=\{e:\|e-x_i\|\le r_{\max}\}$.
Define the hub weighted-degree change $\Delta W(h,e_1,\dots,e_n):= \sum_{i=1}^n \bigl[\varphi(\|e_i-h\|)-\varphi(k_{\min})\bigr].$
Then there exists an admissible maximizer $(h^\star,e_1^\star,\dots,e_n^\star)$ of $\Delta W$ such that
\begin{enumerate}
\item[(i)] $h^\star$ lies on the boundary of $B_0$, and
\item[(ii)] for each $i$, $e_i^\star$ lies on the boundary of $B_i$ along the ray from $h^\star$ through $x_i$, pointing away from $h^\star$.
\end{enumerate}
\end{theorem}

\begin{proof}
Fix the spoke endpoints $e_i$ and define $F(h)=\sum_i \varphi(\|e_i-h\|)$.
Since $\varphi$ is convex and nondecreasing and the Euclidean norm $\|\cdot\|$ is convex, each summand $\varphi(\|e_i-h\|)$ is convex in $h$; hence $F$ is convex.
A convex function on a compact convex set attains its maximum on the boundary, so $h^\star\in\partial B_0$; this proves~(i).
Now fix $h=h^\star$ and consider $G_i(e)=\varphi(\|e-h\|)$ for $e\in B_i$.
Since $\varphi$ is nondecreasing, $G_i$ is maximized by choosing $e$ to maximize $\|e-h\|$, which is achieved at the boundary of $B_i$ along the ray from $h$ through $x_i$; this proves~(ii).
Replacing any admissible configuration by this boundary configuration cannot decrease $\Delta W$.
\end{proof}

\emph{Remark.} The angular separation and packing feasibility that
bounds admissible motif degree is enforced at the level of motif
construction via Assumption (A2) and Theorem \ref{thm:packing-bound},
i.e.~on the unperturbed geometry. Theorem \ref{thm:hub-spoke-extremal}
characterizes where the maximal weighted-degree displacement occurs
within a fixed admissible hub--spoke motif, and therefore imposes no
additional angular constraint on the perturbed configuration:
noise-displaced vertices are not required to preserve pre-noise packing
geometry.

In the remainder of this work we use the shorthand \(G_n\) for the
hub--spoke motifs of degree \(n\) constructed above. For such a motif,
we write \(\delta \mathrm{WD}_{\varepsilon \to 0}(G_n; r) = C_n\, r\) to
denote the maximal weighted-degree displacement in the small-noise
limit, where \(C_n\) is a dimensionless geometric coefficient. The
spectral impact of the extremal configurations identified above is
assessed empirically in Table~\ref{tbl-degree-summary}.

For other members of the admissible metric family
(§\ref{sec-graphmodel}), the structural argument of Theorems
\ref{thm:packing-bound}--\ref{thm:hub-spoke-extremal} carries over, but
the packing geometry of the unit ball and the convexity properties of
the weight kernel are norm-dependent; consequently, both the degree
bound and the boundary-maximiser result must be rederived, yielding a
metric-specific extremal catalogue.

\subsection{Tractable Estimation via Motif
Packing}\label{sec-motif-packing}

The Weyl bound in \eqref{eq:weyl-bound} controls spectral displacement
via the perturbation matrix \(E\), but computing \(\|E\|_2\) or
\(\|E\|_F\) directly requires access to the full noise realization. We
obtain a practical and tight estimate by leveraging the analytical
extremal behaviour of locally fragile motifs.

Let \(\mathcal{M} = \{m_1,\dots,m_p\}\) be a collection of
vertex-disjoint witness motifs embedded in \(G\). Writing the Laplacian
perturbation as a sum over motif-restricted contributions, \[
E \;=\; \sum_{i=1}^{p} E_{m_i} \;+\; E_R,
\] where \(E_{m_i}\) is the Laplacian perturbation contributed by edges
internal to motif \(m_i\)---that is, the off-diagonal entries \((u,v)\)
for edges \(\{u,v\} \subseteq V(m_i)\) and the corresponding diagonal
row sums---and \(E_R\) collects all remaining contributions, including
edges incident to motif vertices but not internal to any single motif.
Since the motifs are vertex-disjoint, the matrices
\(E_{m_1},\dots,E_{m_p}\) have pairwise disjoint row/column support, so
their Frobenius norms add in quadrature:
\begin{equation}\phantomsection\label{eq-frob-quadrature}{
\sum_{i \neq j} \langle E_{m_i}, E_{m_j} \rangle_F = 0,
\qquad\text{hence}\qquad
\Bigl\|\sum_{i=1}^{p} E_{m_i}\Bigr\|_F^2 = \sum_{i=1}^{p} \|E_{m_i}\|_F^2.
}\end{equation} However, the residual \(E_R\) may share row/column
support with individual motif blocks (via diagonal degree entries at
motif vertices with external edges), so quadrature does not extend to
\(E_R\). The triangle inequality yields the always-valid bound
\begin{equation}\phantomsection\label{eq-frob-mixed}{
\|E\|_F \;\le\; \Bigl(\sum_{i=1}^{p} \|E_{m_i}\|_F^2\Bigr)^{1/2} + \|E_R\|_F,
}\end{equation} which retains \(\sqrt{p}\) scaling among the motif terms
while conservatively bounding the residual contribution.

Combining Equation~\ref{eq-frob-mixed} with Weyl's inequality yields the
motif-wise spectral bound \[
|\lambda_k(G')-\lambda_k(G)|
\;\le\; \|E\|_2
\;\le\; \|E\|_F
\;\le\; \Bigl(\sum_{i=1}^{p} \|E_{m_i}\|_F^2\Bigr)^{1/2} + \|E_R\|_F,
\] explicitly relating global spectral displacement to local motif
contributions. In particular, the motif-only part satisfies \[
\max_{i}\|E_{m_i}\|_F
\;\le\; \Bigl(\sum_{i=1}^{p} \|E_{m_i}\|_F^2\Bigr)^{1/2}
\;\le\; \sqrt{p}\;\max_{i}\|E_{m_i}\|_F,
\] showing that, at the level of this Frobenius-based estimate, repeated
occurrences of fragile motifs inflate the motif contribution at most as
\(\sqrt{p}\). Note that the \(\sqrt{p}\) factor is an artifact of the
Frobenius norm bound; since the motif blocks have disjoint support, the
operator norm satisfies \(\|\sum_i E_{m_i}\|_2 = \max_i \|E_{m_i}\|_2\),
which is independent of \(p\). The Frobenius route is retained because
it extends naturally to the mixed bound with \(E_R\).

To make this bound operational, we replace each \(\|E_{m_i}\|_F\) by the
analytically derived worst-case value for its motif type. A maximal
disjoint set of such motifs can be extracted efficiently using a greedy
procedure that prioritizes vertices whose local degree corresponds to
high perturbation sensitivity, marking covered vertices to ensure
disjointness. The tiling enforces vertex-disjointness only; edges
between motifs are permitted, and their perturbative effects are
accounted for in the remainder term \(E_R\). The resulting estimate, \[
\Bigl(\sum_{m_i\in\mathcal{M}} \|E_{m_i,\max}\|_F^2\Bigr)^{1/2},
\] provides a computable, data-driven upper bound on the motif
contribution to spectral fragility, controlled jointly by motif severity
and frequency. Algorithmic details are given in the appendix. Selecting
an optimal disjoint motif cover is computationally intractable in
general. Indeed, the problem of maximizing total motif coverage under
vertex-disjointness constraints reduces to a maximum-weight independent
set problem on an associated conflict graph. Since maximum-weight
independent set is NP-hard, optimal motif tiling is NP-hard as well; a
formal reduction is provided in Appendix
Section~\ref{sec-np-hard-proof}. We therefore rely on the greedy
heuristic above, which is fast, reproducible, and sufficient for tight
practical bounds.

\subsection{Spectral Perturbation Under Constraint
Repair}\label{sec-repair-perturbation}

This section explains how we extend motif-based perturbation
bounds---which are proved under geometric regularity assumptions (e.g.,
\(\varepsilon\)-thickness)---to graphs whose embeddings violate those
assumptions. The core idea is to repair the embedding into a constrained
graph where the motif strategy applies, using local operations that
enforce thickness while preserving coarse geometry.

From a geometric graph \(G=(V,E)\) embedded in \(\mathbb{R}^d\) with
straight-line edges, we build a repaired graph
\(\widetilde G = \mathcal R(G)\) by applying local edits until all
thickness constraints are satisfied. We use the following canonical
operations (illustrated in the Algorithms appendix): edge crossing (if
two edges cross in their interiors, insert a subdivision vertex at the
intersection point and split both edges); vertex collision (if two
vertices violate the identifiability/clearance threshold, fuse them into
a single vertex, merging incident edges and combining weights according
to the model in the Algorithms appendix); and vertex--edge contact (if a
vertex lies on the interior of an edge or violates vertex--edge
clearance, insert a subdivision vertex at the contact and split the
edge).

Each local repair may introduce new constraint violations: splitting an
edge can create a segment that passes close to a third vertex, and
vertex fusion alters the adjacency structure. The repair therefore
proceeds iteratively---via a priority queue in Algorithm 6---until no
violations remain. To see that termination is guaranteed, define the
discrete potential \(\Phi = |V| + C\), where \(C\) is the number of
geometric constraint violations (edge crossings, vertex--vertex
collisions, and vertex--edge proximity violations). Each edit strictly
decreases at least one component: vertex fusion reduces \(|V|\), and
edge subdivision reduces \(C\) for the resolved violation without
increasing \(|V|\). Since \(\Phi\) is nonnegative and integer-valued,
infinite cascading is impossible. We do not claim preservation of global
graph distances or spanner properties; the repair serves to enforce the
geometric constraints required for subsequent motif analysis. In
particular, spectral stability does not follow from metric distortion
bounds, and edge weights are not generally additive under subdivision.

Algorithm 6 specifies the repair procedure abstractly; a general-purpose
implementation for arbitrary \(d\)-dimensional collision resolution is
not provided. The experiments in Section~\ref{sec-ctpl-noise-results}
use a hand-crafted repair tailored to the reference graph, where
constraint violations are few and geometrically transparent. In general,
automated \(k\)-dimensional collision resolution is non-trivial and lies
outside the scope of this work: naive implementations risk violating
modeling-domain constraints or introducing pathological collapses,
particularly when violations are dense or clustered. For example, if all
vertices lie within \(\varepsilon\) of one another, contraction
collapses the graph to a single vertex with arbitrarily large spectral
deviation; similarly, a high-degree hub with many crossing edges can
cause combinatorial blow-up under naive per-crossing subdivision. The
perturbation bounds are therefore meaningful in the sparse-violation
regime assumed in the experiments.

The repaired graph (constrained surrogate) \(\widetilde G\) has vertex
set \(V_0 \cup V_s\), where \(V_0\) are the original vertices (possibly
after fusion) and \(V_s\) are deterministic subdivision vertices
introduced by repair. By construction, \(\widetilde G\) satisfies the
geometric constraints needed for the motif-tiling bounds in
Section~\ref{sec-motif-packing}.

To clarify the scope of the subsequent perturbation bound, observe that
the total spectral effect of noise on the original graph decomposes as
\[
L_{G'} - L_G
\;=\;
\bigl(L_{G'} - L_{\widetilde G'}\bigr)
\;+\;
\bigl(L_{\widetilde G'} - L_{\widetilde G}\bigr)
\;+\;
\bigl(L_{\widetilde G} - L_G\bigr).
\] The theorem below controls the middle term---spectral perturbation of
the surrogate under vertex noise. The first and third terms represent
modeling distortion introduced by repair. Each vertex contraction
modifies at most \(O(d_{\max})\) rows of the Laplacian, and each edge
split modifies \(O(1)\) rows, so the total repair distortion scales with
the number of edits and the maximum degree. This distortion remains
small when repairs are sparse---the regime assumed in the
experiments---but may degrade under adversarial configurations; a
non-trivial formal bound would require accounting for pathological cases
and is not attempted here.

\subsubsection{A degree-controlled Laplacian perturbation bound on a
repaired
graph}\label{a-degree-controlled-laplacian-perturbation-bound-on-a-repaired-graph}

We consider i.i.d.~\ac{CTPL} vertex noise acting on the original
vertices only. Let \(L_{\widetilde G}\) be the combinatorial Laplacian
of \(\widetilde G\) and \(L_{\widetilde G'}\) the Laplacian after
perturbing the coordinates of vertices in \(V_0\). Write the Laplacian
perturbation as \[
\Delta_{\widetilde G} := L_{\widetilde G'} - L_{\widetilde G}.
\]

\begin{theorem}[Degree-controlled Frobenius bound and spectral shift control]\label{thm:repair-spectral-bound}
Let $H = \widetilde G$ be the repaired graph with vertex set $V_0\cup V_s$, where noise acts only on $V_0$.
Let $\delta w_{uv}$ denote the change in edge weight on $\{u,v\}\in E(H)$ induced by the noisy coordinates.
Define
$$
S(H) := \sum_{\{u,v\}\in E(H)} (\delta w_{uv})^2,
\qquad
\Delta(H) := \max_{u\in V_0\cup V_s} \deg_H(u).
$$
Then the Laplacian perturbation satisfies the Frobenius estimate
\begin{equation}\label{eq:frob-degree-bound}
\|\Delta_H\|_F^2
\;\le\; \bigl(2 + 2\,\Delta(H)\bigr)\, S(H).
\end{equation}
Consequently, writing $S_2(L,L') := \bigl(\sum_{i=1}^n (\lambda_i(L')-\lambda_i(L))^2\bigr)^{1/2}$ for the $\ell_2$ spectral shift, we have
\begin{equation}\label{eq:spectralshift-from-frob}
S_2(L_H,L_{H'}) \le \|\Delta_H\|_F
\;\le\; \sqrt{2 + 2\,\Delta(H)} \;\sqrt{S(H)}.
\end{equation}
\end{theorem}

\begin{proof}
Write $\Delta_H = L_{H'}-L_H$.
For a weighted Laplacian, off-diagonal entries satisfy $(L_H)_{uv}=-w_{uv}$ for $\{u,v\}\in E(H)$, hence off-diagonal perturbations contribute
$$
\sum_{u\ne v} (\Delta_H)_{uv}^2
= 2 \sum_{\{u,v\}\in E(H)} (\delta w_{uv})^2
= 2\,S(H).
$$
Diagonal entries satisfy $(L_H)_{uu}=\sum_{v\sim u} w_{uv}$, so the diagonal perturbation at $u$ is
$\sum_{v\sim u}\delta w_{uv}$, and
$$
\sum_u (\Delta_H)_{uu}^2
= \sum_{u\in V_0\cup V_s} \Bigl(\sum_{v\sim u}\delta w_{uv}\Bigr)^2.
$$
For each $u\in V_0\cup V_s$, Cauchy--Schwarz yields
$$
\Bigl(\sum_{v\sim u}\delta w_{uv}\Bigr)^2
\le \deg_H(u)\sum_{v\sim u}(\delta w_{uv})^2
\le \Delta(H)\sum_{v\sim u}(\delta w_{uv})^2.
$$
Summing over $u\in V_0\cup V_s$ and re-indexing edge terms gives
$$
\sum_{u\in V_0\cup V_s}\Bigl(\sum_{v\sim u}\delta w_{uv}\Bigr)^2
\le 2\,\Delta(H)\sum_{\{u,v\}\in E(H)}(\delta w_{uv})^2
= 2\,\Delta(H)\,S(H).
$$
Combining diagonal and off-diagonal contributions yields \eqref{eq:frob-degree-bound}.
The spectral-shift inequality \eqref{eq:spectralshift-from-frob} follows from the Hoffman--Wielandt theorem
\cite{HoffmanWielandt1953}.
\end{proof}

\subsection{Stochastic spectral metrics for observability and
separability}\label{sec-scssi}

To quantify how vertex noise affects spectral measurements, we do not
introduce new spectral distances. Instead, we formalize moment-based
summaries of existing spectral distances in order to assess (i) whether
spectral differences are observable under noise, and (ii) whether
different noise regimes can be distinguished from spectral data alone.
Fix a graph \(G\), a noise model \(\mathcal N\), and a spectral distance
\(d_{SD}\). Let \(D := d_{SD}\bigl(G,\mathcal N(G)\bigr)\) denote the
induced spectral-distance random variable.

\subsubsection{Stochastic co-spectrality
(SC)}\label{stochastic-co-spectrality-sc}

We define the stochastic co-spectrality moments as the raw moments of
\(D\): \[
\mathrm{SC}_k(G;\mathcal N) := \mathbb E[D^k], \qquad k=1,2,\ldots
\] whenever they exist. In particular, \[
\mu_D := \mathrm{SC}_1(G;\mathcal N),
\qquad
\sigma_D^2 := \mathrm{SC}_2(G;\mathcal N)-\mathrm{SC}_1(G;\mathcal N)^2 .
\] These quantities summarize the expected spectral displacement induced
by noise and its variability. For untempered heavy-tailed noise models,
such moments need not exist; under the calibrated CTPL model, all
moments required below are finite. A central use of \(\mathrm{SC}\) is
to assess noise observability. Let
\(D^\star := d_{SD}\bigl(\mathcal N(G),\mathcal N(G')\bigr)\) be an
observed spectral distance between two noisy graphs. Using Cantelli's
inequality \citep{Cantelli1910}, if \(\mu_D\) and \(\sigma_D^2\) exist
then
\(\Pr\!\left[D \ge \mu_D + \sigma_D \sqrt{\tfrac{1-\alpha}{\alpha}}\right] \le \alpha.\)
We therefore say that \(D^\star\) is noise-observable at level
\(\alpha\) if
\(\frac{D^\star-\mu_D}{\sigma_D} > \sqrt{\tfrac{1-\alpha}{\alpha}}.\)

\subsubsection{Observable SC from noisy data (weak-oracle
setting)}\label{observable-sc-from-noisy-data-weak-oracle-setting}

When the noise-free graph \(G\) is unavailable, \(\mathrm{SC}\) can
still be estimated from paired noisy observations. Let
\(G_i,G_j\sim\mathcal N(G)\) be independent perturbations. Define \[
\mathrm{SC}_{\mathrm{obs}}(G;\mathcal N)
:= \mathbb E\!\left[d_{SD}(G_i,G_j)\right].
\] By the triangle inequality, \[
\mathbb E\!\left[|D_i-D_j|\right]
\;\le\;
\mathrm{SC}_{\mathrm{obs}}(G;\mathcal N)
\;\le\;
2\,\mathbb E[D],
\] so \(\mathrm{SC}_{\mathrm{obs}}\) provides an observable proxy for
spectral compactness under noise, with finite variance whenever
\(\mu_D\) and \(\sigma_D^2\) exist.

\subsubsection{Stochastic spectral separation index
(S3I)}\label{stochastic-spectral-separation-index-s3i}

To compare two noise models \(\mathcal N_a\) and \(\mathcal N_b\) acting
on the same graph \(G\), let \[
D^{(a)} := d_{SD}\bigl(G,\mathcal N_a(G)\bigr),
\qquad
D^{(b)} := d_{SD}\bigl(G,\mathcal N_b(G)\bigr),
\] with means \(\mu_a,\mu_b\) and variances \(\sigma_a^2,\sigma_b^2\).
We define the stochastic spectral separation index as \[
\mathrm{S3I}(G;\mathcal N_a,\mathcal N_b)
:=
\frac{|\mu_a-\mu_b|}
{\sqrt{(\sigma_a^2+\sigma_b^2)/2}} .
\] This quantity is an effect-size-style summary, analogous to Cohen's
\(d\), measuring how distinguishable the average spectral impacts of two
noise regimes are relative to their pooled variability. In practice,
\(\mathrm{S3I}\) is estimated from empirical samples (either
strong-oracle or weak-oracle), and is used descriptively rather than as
a formal separability guarantee. In empirical settings, both
\(\mathrm{SC}\) and \(\mathrm{S3I}\) are estimated from finite samples
of noisy realizations. To assess uncertainty we compute \(\mathrm{S3I}\)
independently over \(N\) replicated Monte Carlo runs, each drawing fresh
noise realisations, and report the \((1-\alpha)\) percentile interval of
the resulting \(\mathrm{S3I}\) distribution. If this interval excludes
zero, the two noise regimes are deemed empirically separable under the
\((1-\alpha)\) percentile criterion. This procedure is used throughout
the experimental results to assess robustness of spectral
distinguishability.

\subsection{Extent of vertex--edge noise distribution
divergence}\label{sec-vn-noise-results}

\begin{figure}[htbp]

\centering{

\includegraphics[width=\linewidth,height=3in,keepaspectratio]{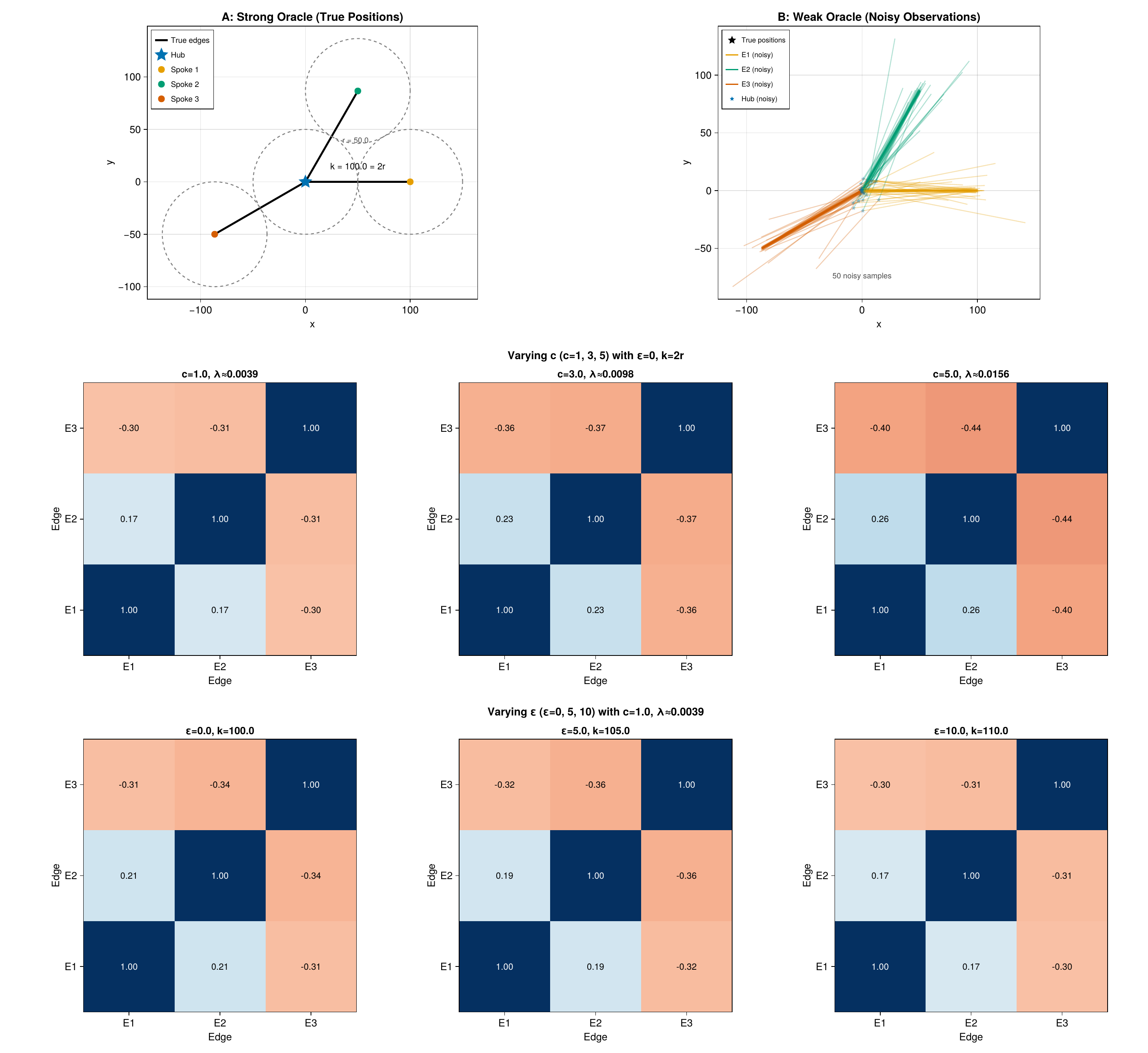}

}

\caption{\label{fig-vertex-noise-is-not-edge-noise}Vertex noise induces
correlated edge perturbations. (A) Strong oracle: an asymmetric
three-spoke hub motif at \(k=2r_{\max}\) with ground-truth positions
(blue star = hub, colored circles = spokes) and perturbation balls (gray
dashed). (B) Weak oracle: 50 noisy realisations overlaid; faded black
lines and markers show ground-truth positions for reference, coloured
semi-transparent lines show perturbed edges. Bottom rows: empirical
edge-length correlation matrices (\(500\) samples each). Row 2 varies
the scale parameter \(c\) at fixed \(\varepsilon=0\) (with \(\lambda\)
recalibrated per \(c\)); Row 3 varies \(\varepsilon\) at fixed \(c=1\).
Off-diagonal correlations are consistently nonzero, confirming that
vertex noise induces structured edge dependencies even in minimal
configurations.}

\end{figure}%

Figure~\ref{fig-vertex-noise-is-not-edge-noise} illustrates a
fundamental distinction between vertex- and edge-based noise models
using an asymmetric three-spoke hub motif (\(k=100\), \(r_{\max}=50\),
spokes at \(0^\circ\), \(60^\circ\), \(210^\circ\); \(500\) noise
samples per condition, calibration tolerance \(\delta=0.05\)). Even for
this minimal configuration, independent vertex perturbations induce
correlated and geometry-dependent edge-weight fluctuations. This
confirms that vertex noise cannot, in general, be reduced to independent
or identically distributed edge noise, motivating the motif-level
analysis developed in the Methods section.

\subsection{\texorpdfstring{Empirical spectral signatures of \ac{CTPL}
noise}{Empirical spectral signatures of  noise}}\label{sec-ctpl-noise-results}

\begin{figure}[htbp]

\centering{

\pandocbounded{\includegraphics[keepaspectratio]{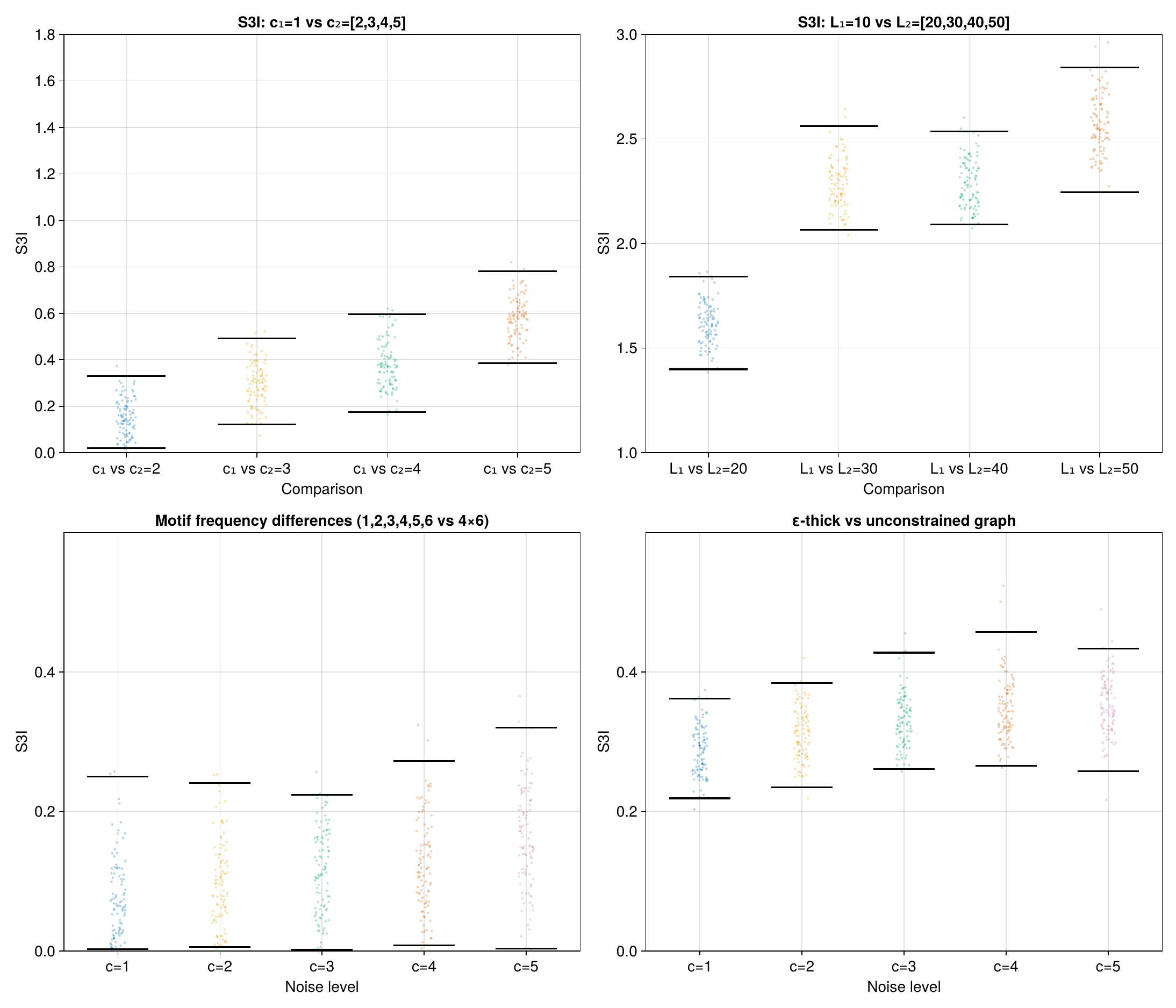}}

}

\caption{\label{fig-s3-results}Motif-informed graph-level sensitivity
and repair effects. Each panel reports \ac{S3I} values from \(100\)
independent Monte Carlo runs (\(200\) noise samples each), with \(95\%\)
percentile confidence intervals (black bars). A checkmark indicates that
the interval excludes zero, i.e.~\ac{S3I} is consistently positive
across runs. (A) Separation induced by changes in the scale parameter
\(c\) at fixed \(r_{\max}\). (B) Separation induced by changing
\(r_{\max}\) at fixed \(c\). (C) Separation between graphs with
different motif compositions (heterogeneous \(1\)--\(6\) versus uniform
\(4\times 6\)) under identical noise. (D) Separation between an
unconstrained graph and its \(\varepsilon\)-thick repaired counterpart.
Note that each comparison is a pairwise test at a single operating
point; \ac{S3I} values should not be compared across operating points
without further testing.}

\end{figure}%

We evaluate \ac{S3I} on the reference graph of
Figure~\ref{fig-motifgraph} (see Appendix) (\(k=100\), inter-motif
spacing \(300\); \(28\) vertices, \(27\) edges), using calibrated
\ac{CTPL} noise with tolerance \(\delta=0.05\). Each comparison
comprises \(100\) independent runs of \(200\) noise samples, with
\(95\%\) percentile intervals reported. Panel A compares \(c_1=1\)
against \(c_2\in\{2,3,4,5\}\) at fixed \(r_{\max}=40\); Panel B compares
\(r_{\max}=10\) against \(r_{\max}\in\{20,30,40,50\}\) at fixed \(c=1\);
Panel C compares the heterogeneous reference graph against a uniform
(\(6\times\) degree-\(4\)) variant at \(c\in\{1,\dots,5\}\); Panel D
compares an unconstrained graph (hub removed, crossing edges) against
its \(\varepsilon\)-thick repaired counterpart at \(c\in\{1,\dots,5\}\).

Figure~\ref{fig-s3-results} demonstrates that \ac{S3I} detects multiple,
qualitatively distinct sources of spectral variation. Varying the
\ac{CTPL} scale parameter \(c\) (Panel A) or the maximal jump radius
\(r_{\max}\) (Panel B) produces clear spectral separation, with all
confidence intervals well above zero, indicating sensitivity to noise
parameters alone. Changing motif composition while holding global
structure fixed (Panel C) yields a weaker but detectable separation:
\ac{S3I} is consistently positive, though the lower confidence bound
approaches zero at low noise, indicating that separability diminishes as
perturbation magnitude decreases. Similarly, enforcing
\(\varepsilon\)-thickness via local repair (Panel D) produces a bounded
spectral shift consistent with the degree-controlled perturbation
framework of Section~\ref{sec-motif-packing}, with lower bounds close to
zero at low \(c\). Each panel is a pairwise comparison at a given
operating point; no cross-comparison between operating points
(e.g.~across \(c\) values) is tested.

Together, these experiments confirm that \ac{S3I} provides a practical,
observable summary of how noise parameters, geometric constraints, and
local graph structure jointly shape spectral variability under vertex
noise.

\subsection{Motif geometry
sensitivity}\label{sec-motif-geometry-results}

\begin{figure}[htbp]

\centering{

\includegraphics[width=0.9\linewidth,height=\textheight,keepaspectratio]{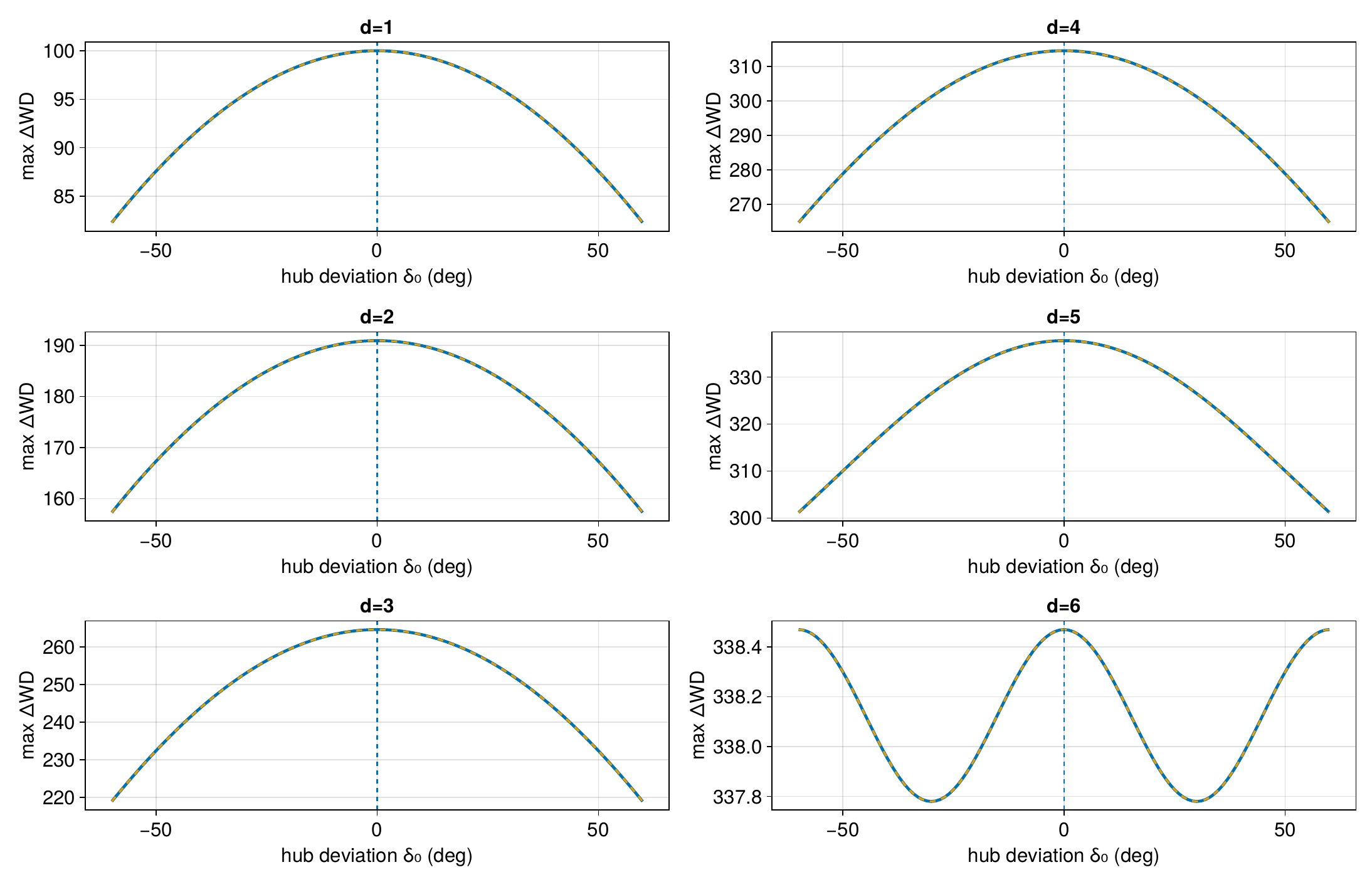}

}

\caption{\label{fig-motifs-optimality}Sensitivity of witness motifs to
angular perturbations. For each motif degree \(n=1,\dots,6\), the
maximum change in weighted degree \(\Delta \mathrm{WD}\) is shown as a
function of the hub displacement angle relative to the optimal bisector
axis of the unoccupied arc. For \(n \le 5\), the extremum is unique; for
\(n=6\), hexagonal symmetry induces multiple equivalent maxima. The
admissible angular range is determined by the \(\varepsilon\)-thickness
constraint.}

\end{figure}%

Figure~\ref{fig-motifs-optimality} empirically validates the geometric
extremality assumptions underlying the motif analysis. For each degree
\(n=1,\dots,6\), hub--spoke motifs are constructed at \(k=100\) with
\(r=50\); the hub displacement angle is swept over \(121\) grid points
in \([-60^\circ,60^\circ]\) relative to the bisector, with \(50\) random
spoke-angle perturbations per grid point. For all degrees, maximal
sensitivity is attained at the canonical configurations identified
analytically, and angular deviations strictly reduce the induced
weighted-degree change. This confirms that the witness motifs used
throughout the paper are not artefacts of a particular alignment but
represent genuine local maxima under the admissible geometric
constraints.

\subsection{Spectral tightness of witness
motifs}\label{sec-motif-tightness-results}

Table~\ref{tbl-degree-summary} reports maximal weighted-degree
perturbations for each witness motif in the extremal regime (\(k=2r\),
\(\varepsilon\to 0\), \(r=1\)), together with the corresponding observed
maximum Laplacian eigenvalue shift and Weyl upper bound
(\(\|L'-L\|_2\)); the ratio column gives the tightness of the Weyl bound
(shift/bound). Low-degree motifs (degrees 1 and 2) exhibit numerical
Weyl tightness (ratio \(=1\)), while higher-degree motifs show
increasing slack, indicating that spectral misalignment---rather than
noise magnitude alone---limits tightness. These results justify the
motif ranking used in the packing-based fragility bounds and support the
use of low-degree motifs as sharp spectral witnesses in practice.

\begin{longtable}[]{@{}
  >{\raggedright\arraybackslash}p{(\linewidth - 8\tabcolsep) * \real{0.1739}}
  >{\raggedright\arraybackslash}p{(\linewidth - 8\tabcolsep) * \real{0.2391}}
  >{\raggedright\arraybackslash}p{(\linewidth - 8\tabcolsep) * \real{0.2391}}
  >{\raggedright\arraybackslash}p{(\linewidth - 8\tabcolsep) * \real{0.1957}}
  >{\raggedright\arraybackslash}p{(\linewidth - 8\tabcolsep) * \real{0.1522}}@{}}
\caption{Maximum relative change in weighted degree for maximally
affected subgraphs under \ac{CTPL} noise, together with observed
spectral displacement and Weyl bounds. Here \(r\) denotes \(r_{\max}\)
(the extremal regime sets \(k = 2r_{\max}\),
\(\varepsilon \to 0\)).}\label{tbl-degree-summary}\tabularnewline
\toprule\noalign{}
\begin{minipage}[b]{\linewidth}\raggedright
Degree
\end{minipage} & \begin{minipage}[b]{\linewidth}\raggedright
\(\delta\) WD \(\varepsilon \to 0\)
\end{minipage} & \begin{minipage}[b]{\linewidth}\raggedright
Actual Shift
\end{minipage} & \begin{minipage}[b]{\linewidth}\raggedright
Weyl Limit
\end{minipage} & \begin{minipage}[b]{\linewidth}\raggedright
Ratio
\end{minipage} \\
\midrule\noalign{}
\endfirsthead
\toprule\noalign{}
\begin{minipage}[b]{\linewidth}\raggedright
Degree
\end{minipage} & \begin{minipage}[b]{\linewidth}\raggedright
\(\delta\) WD \(\varepsilon \to 0\)
\end{minipage} & \begin{minipage}[b]{\linewidth}\raggedright
Actual Shift
\end{minipage} & \begin{minipage}[b]{\linewidth}\raggedright
Weyl Limit
\end{minipage} & \begin{minipage}[b]{\linewidth}\raggedright
Ratio
\end{minipage} \\
\midrule\noalign{}
\endhead
\bottomrule\noalign{}
\endlastfoot
1 & \(2r\) & 4.0000 & 4.0000 & 1.0000 \\
2 & \(3.82r\) & 5.7279 & 5.7279 & 1.0000 \\
3 & \(5.29r\) & 7.0653 & 7.0768 & 0.9984 \\
4 & \(6.29r\) & 7.9030 & 7.9523 & 0.9938 \\
5 & \(6.76r\) & 8.2030 & 8.3373 & 0.9839 \\
6 & \(6.77r\) & 8.0702 & 8.3460 & 0.9670 \\
\end{longtable}

\section{Discussion and Conclusions}\label{sec-discussion}

This work formalises vertex noise in geometric graphs through calibrated
tempered power-law (CTPL) perturbations and introduces two complementary
stochastic spectral quantities: stochastic co-spectrality (\ac{SC}),
which captures the distributional footprint of spectral distances under
noise, and the stochastic spectral separation index (\ac{S3I}), which
summarises the distinguishability of noise regimes from spectral data
alone. Together with motif-based perturbation analysis, these quantities
provide a principled bridge between local geometric noise mechanisms and
global spectral effects.

A central insight of the paper is that vertex noise induces structured,
correlated, and geometry-dependent edge perturbations that cannot, in
general, be reduced to independent edge-noise models. This motivates the
use of local geometric witness motifs as the correct scale at which
worst-case and aggregate spectral effects concentrate. By combining
sharp local extremal analysis with Weyl-type and Frobenius bounds, we
obtain tractable estimates of global spectral fragility that depend
explicitly on motif frequency and degree structure rather than on opaque
global constants. The resulting framework cleanly separates where
fragility originates (local geometry) from how it propagates (spectral
perturbation theory).

From an algorithmic perspective, the greedy motif-tiling strategy used
here is deliberately simple. The associated optimal tiling problem is
NP-hard, so exact solutions are infeasible at scale; nevertheless, the
empirical results indicate that even coarse tilings capture the dominant
contributors to spectral instability. Designing faster or
approximation-guaranteed motif extraction schemes is a natural direction
for future work. The theoretical framework---Weyl-type bounds, Frobenius
decomposition, and the motif-packing strategy---holds in any dimension
under assumptions (A1)--(A4), but the specific witness catalogue and
empirical tightness constants (Table~\ref{tbl-degree-summary}) are
derived and validated in two dimensions. Extending the motif catalogue
and sharp constants to higher-dimensional settings, where spatial
saturation constraints tighten and the family of extremal configurations
diversifies, will likely require new structural insights.

While our analysis is restricted to norm-induced distances, a
significant part of the community works with similarity-based metrics,
where proximity increases, rather than decreases the weight of an edge.
Due to space limitations we cannot fully explore this direction here.
Extending the present bounds to similarity kernels is non-trivial:
decreasing or near-singular kernels (e.g.~inverse-distance) amplify
small post-noise separations, breaking the uniform perturbation
envelopes used above unless additional post-noise separation or kernel
regularization is imposed. The obstruction is therefore analytic rather
than computational. We note that our accompanying open-source packages
(Section~\ref{sec-algorithms}) support arbitrary kernels, so
practitioners can empirically assess the impact of their chosen kernel
on spectral fragility and separability in their own settings.

Beyond methodology, the framework has direct relevance for spatial
networks arising in imaging and biological reconstruction, such as
mitochondrial and endoplasmic reticulum networks, filament skeletons,
and single-molecule localisation graphs. In these settings,
perturbations are naturally vertex-local, heavy-tailed, and
geometry-constrained. The quantities introduced here---\ac{SC} and
\ac{S3I}---offer interpretable summaries of spectral robustness and
separability that remain meaningful even when the underlying noise
distributions are non-Gaussian and only weak-oracle observations are
available. In particular, estimating the noise-induced spectral envelope
provides a computable noise floor that can replace domain-specific
simulations or edge-independent approximations, enabling practitioners
to delineate genuine structural change from localisation-induced
spectral artefacts and thereby reduce both false positives and overly
conservative false-negative thresholds. The extension to unconstrained
embeddings via repair ensures this capability is not limited to
biophysically motivated networks, and the accompanying open-source
packages provide domain experts with ready-to-use tooling for spectral
noise quantification on their own geometric configurations.

A longer-term implication is the possibility of inversion: given noisy
spectral observations and partial geometric priors, can one constrain or
recover parameters of the underlying vertex-noise process using
motif-level approximations and stochastic spectral summaries? Addressing
such inverse problems, alongside scalable motif detection in higher
dimensions, would further connect the present theory to data-driven
applications. Taken together, the results of this work establish
motif-aware stochastic spectral analysis as a viable and flexible lens
for studying vertex noise in geometric graphs.

\section{References}\label{references}

\renewcommand{\bibsection}{}

\input{paper.bbl}
\section*{Acknowledgments}
F.S.\ and B.C.\ acknowledge funding from a UKRI Future Leaders Fellowship MR/T043571/1.

\newpage

\section{Supplementary Material
(Appendix)}\label{supplementary-material-appendix}

\subsection{Notation reference}\label{sec-notation}

{\footnotesize

\begin{longtable}[]{@{}
  >{\raggedright\arraybackslash}p{(\linewidth - 4\tabcolsep) * \real{0.2424}}
  >{\raggedright\arraybackslash}p{(\linewidth - 4\tabcolsep) * \real{0.3939}}
  >{\raggedright\arraybackslash}p{(\linewidth - 4\tabcolsep) * \real{0.3636}}@{}}
\caption{Principal notation used in this paper. Standard graph-theoretic
symbols (\(G\), \(V\), \(E\), \(A\), \(D\), \(L\)) follow the
definitions in
Section~\ref{sec-defs}.}\label{tbl-notation}\tabularnewline
\toprule\noalign{}
\begin{minipage}[b]{\linewidth}\raggedright
Symbol
\end{minipage} & \begin{minipage}[b]{\linewidth}\raggedright
Description
\end{minipage} & \begin{minipage}[b]{\linewidth}\raggedright
Defined in
\end{minipage} \\
\midrule\noalign{}
\endfirsthead
\toprule\noalign{}
\begin{minipage}[b]{\linewidth}\raggedright
Symbol
\end{minipage} & \begin{minipage}[b]{\linewidth}\raggedright
Description
\end{minipage} & \begin{minipage}[b]{\linewidth}\raggedright
Defined in
\end{minipage} \\
\midrule\noalign{}
\endhead
\bottomrule\noalign{}
\endlastfoot
\(k\) & Ground-truth edge length (strong oracle) &
Section~\ref{sec-graphmodel}, Figure~\ref{fig-graphmodel} \\
\(k'\) & Noise-perturbed edge length (weak oracle) &
Section~\ref{sec-tailweights} \\
\(k_{\min}\) & Minimum admissible edge length, \(2r_{\max}+\varepsilon\)
& Section~\ref{sec-graphmodel} (A2), Figure~\ref{fig-graphmodel} \\
\(\varepsilon\) & Clearance / \(\varepsilon\)-thickness separation &
Section~\ref{sec-graphmodel} (A1) \\
\(r_{\max}\) & Maximum admissible jump radius (clipping bound) &
Section~\ref{sec-ctpl} \\
\(r_a, r_b\) & Radial displacement magnitudes per vertex &
Figure~\ref{fig-graphmodel} \\
\(c\) & Scale parameter of the \ac{CTPL} distribution &
Section~\ref{sec-levy-base}, Section~\ref{sec-ctpl} \\
\(\lambda\) & Exponential tempering parameter &
Section~\ref{sec-ctpl} \\
\(\delta\) & Tail tolerance for calibration & Section~\ref{sec-ctpl} \\
\(N_D\) & Doubling constant of the ambient metric &
Section~\ref{sec-graphmodel} \\
\(\varphi\) & Radial weight kernel (edge weight as function of distance)
& Section~\ref{sec-degreeproof} \\
\(\rho\) & Inter-vertex distance (scalar argument to weight function) &
Section~\ref{sec-graphmodel} (A4) \\
\(w_{\max}, L_w\) & Weight bound and Lipschitz constant of \(\varphi\) &
Section~\ref{sec-graphmodel} (A4) \\
\(\theta\) & Minimum angular separation between neighbours at a hub &
Section~\ref{sec-graphmodel} (A2) \\
\(\tau_d\) & Kissing number in dimension \(d\) (\(\tau_2=6\)) &
Section~\ref{sec-degreeproof} \\
\(\mathrm{WD}_G(v)\) & Weighted degree of vertex \(v\) &
Section~\ref{sec-degreeproof} \\
\(\delta_{\mathrm{rel}}\mathrm{WD}(v)\) & Relative weighted-degree
change under perturbation & Section~\ref{sec-degreeproof} \\
\(C_n\) & Geometric sensitivity coefficient for degree-\(n\) motif &
Section~\ref{sec-degreeproof} \\
\(G_n\) & Hub--spoke witness motif of degree \(n\) &
Section~\ref{sec-degreeproof} \\
\(d_{\mathrm{SD}}\) & Normalised spectral distance (Wasserstein-2 based)
& Section~\ref{sec-defs} \\
\(S_2(L,L')\) & \(\ell_2\) spectral shift &
Section~\ref{sec-repair-perturbation} \\
\(\mathrm{SC}_k\) & Stochastic co-spectrality (\(k\)-th moment of
\(d_{\mathrm{SD}}\)) & Section~\ref{sec-scssi} \\
\(\mathrm{S3I}\) & Stochastic spectral separation index &
Section~\ref{sec-scssi} \\
\(S(H)\) & Sum of squared edge-weight changes on graph \(H\) &
Section~\ref{sec-repair-perturbation} \\
\(\Delta(H)\) & Maximum vertex degree in graph \(H\) &
Section~\ref{sec-repair-perturbation} \\
\end{longtable}

}

\subsection{Algorithms and reproduction}\label{sec-algorithms}

Source code supporting this paper is structured as 2 Julia packages and
supporting scripts:

\begin{itemize}
  \item \href{https://github.com/systems-mechanobiology/DynamicGeometricGraphs.jl}{DynamicGeometricGraphs.jl}:
    geometric graph constructions, repair algorithms, motif detection and tiling.\\
    {\small\url{https://github.com/systems-mechanobiology/DynamicGeometricGraphs.jl}}
  \item \href{https://github.com/systems-mechanobiology/CalibratedTemperedPowerLaw.jl}{CalibratedTemperedPowerLaw.jl}:
    implementation of the \ac{CTPL} distribution, sampling algorithms, S3I, SC metrics, and noise perturbation routines.\\
    {\small\url{https://github.com/systems-mechanobiology/CalibratedTemperedPowerLaw.jl}}
  \item \href{https://github.com/systems-mechanobiology/GeometricVertexNoiseSpectra}{GeometricVertexNoiseSpectra}:
    scripts to reproduce figures and experiments in the paper, including the geometric motifs and maximal configurations and angular maximizers.\\
    {\small\url{https://github.com/systems-mechanobiology/GeometricVertexNoiseSpectra}}
\end{itemize}

\begin{figure}[htbp]

\centering{

\pandocbounded{\includegraphics[keepaspectratio]{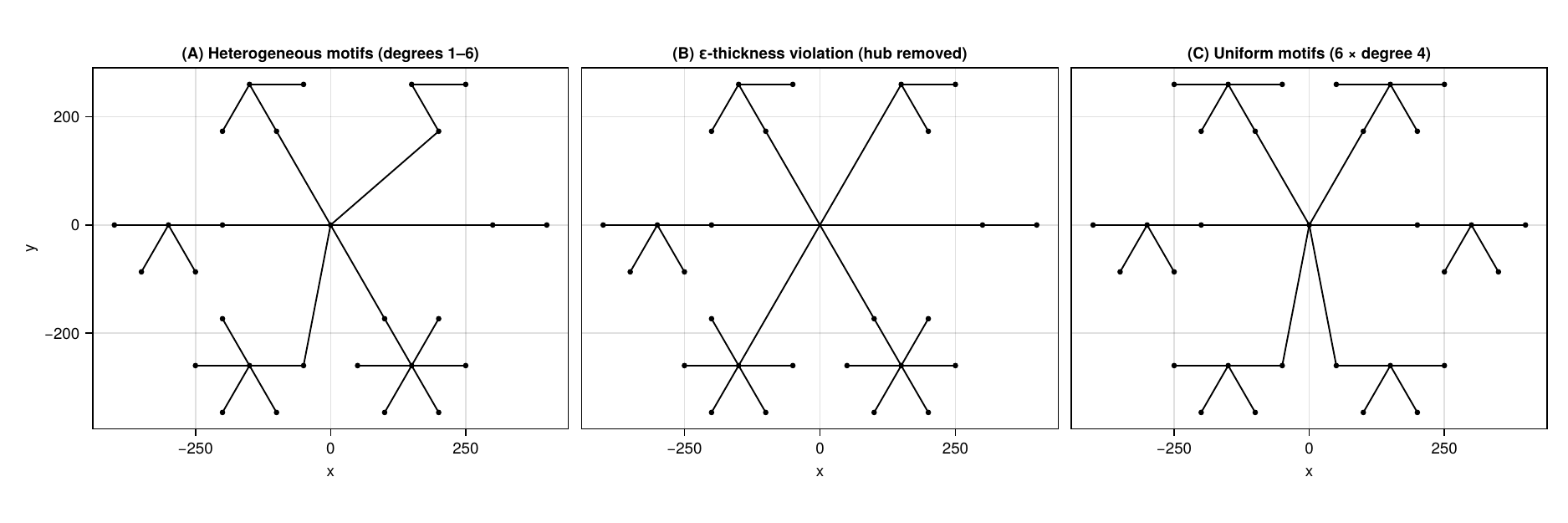}}

}

\caption{\label{fig-motifgraph}Graph constructions used for
motif-frequency and repair experiments. Three geometric graph
constructions sharing the same spatial scaffold (edge length \(k=100\),
inter-motif spacing \(300\); arbitrary units). (A) Motif-complete
reference graph realizing all six witness motifs under the
\(\varepsilon\)-thickness constraint, used as the baseline construction
in Figure~\ref{fig-s3-results}. (B) Variant violating the
\(\varepsilon\)-thickness constraint by removal of the central hub,
illustrating the geometric collapse excluded by the model assumptions.
(C) Frequency-controlled construction with uniform motif composition
(six copies of the degree-4 motif), used to isolate the effect of motif
frequency independently of geometry. Exact construction parameters are
given in the reproduction scripts.}

\end{figure}%

\SetAlFnt{\scriptsize}
\SetAlCapFnt{\scriptsize}
\SetAlCapNameFnt{\scriptsize}
\setlength{\algomargin}{0.5em}
\SetAlgoSkip{}
\begin{algorithm2e}[H]
\caption{Sample from \ac{CTPL} distribution}
\label{alg:sample-tl}
\KwIn{Scale $c > 0$, tempering $\lambda \geq 0$, location $\mu$, sample limit $L$}
\KwOut{Sample $x$ from $\mathrm{CTPL}(c, \lambda, \mu)$, $x \leq L$}
\While{True}{
    Draw $z \sim \mathcal{N}(0, 1)$;\\
    $x_{\text{cand}} \leftarrow \mu + \frac{c}{z^2}$;\\
    Compute proposal PDF: $p_{\text{prop}} \leftarrow \mathrm{LevyPDF}(x_{\text{cand}}, c, \mu)$;\tcp*{Sec~2.2.1}
    \If{$p_{\text{prop}} \leq 0$ or not finite}{continue;}
    Compute target PDF: $p_{\text{target}} \leftarrow \mathrm{TemperedLevyPDF}(x_{\text{cand}}, c, \lambda, \mu)$;\tcp*{Sec~2.2.2}
    $r \leftarrow \frac{p_{\text{target}}}{p_{\text{prop}}}$;\\
    Draw $u \sim \mathrm{Uniform}(0, 1)$;\\
    \If{$u < r$}{return $\min(L, x_{\text{cand}})$;}
}
\end{algorithm2e}

\begin{algorithm2e}[H]
\caption{Autotune $\lambda$ for Tail Probability}
\label{alg:autotune-lambda}
\KwIn{Scale $c$, location $\mu$, sample limit $L$, target tail probability $\delta$, tolerance $\epsilon$, max iterations $N$}
\KwOut{$\lambda$ such that $P(X > L) < \delta$}
Set $\lambda_{\text{low}} \leftarrow 0$, $\lambda_{\text{high}} \leftarrow 2$ (sufficient for the parameter ranges considered here);\\
\For{$i \leftarrow 1$ to $N$}{
    $\lambda_{\text{mid}} \leftarrow (\lambda_{\text{low}} + \lambda_{\text{high}})/2$;\\
    $p_{\text{tail}} \leftarrow \text{estimate-tail}(\lambda_{\text{mid}}, \mu, c, M)$;\\
    \If{$|p_{\text{tail}} - \delta| < \epsilon$}{return $\lambda_{\text{mid}}$;}
    \ElseIf{$p_{\text{tail}} > \delta$}{$\lambda_{\text{low}} \leftarrow \lambda_{\text{mid}}$;}
    \Else{$\lambda_{\text{high}} \leftarrow \lambda_{\text{mid}}$;}
}
\KwRet{best $\lambda_{\mathrm{mid}}$}\;
\end{algorithm2e}

\begin{algorithm2e}[H]
\caption{Estimate Tail Probability by Monte Carlo}
\label{alg:estimate-tail}
\KwIn{Scale $c$, tempering $\lambda$, location $\mu$, threshold $L$, number of samples $M$}
\KwOut{Estimated tail probability $p_{\text{tail}}$}
$count \leftarrow 0$;\\
\For{$i \leftarrow 1$ to $M$}{
    Draw $x \sim \mathrm{CTPL}(c, \lambda, \mu)$;\\
    \If{$x > L$}{$count \leftarrow count + 1$;}
}
$p_{\text{tail}} \leftarrow count / M$;\\
return $p_{\text{tail}}$;
\end{algorithm2e}

\begin{algorithm2e}[H]
\caption{Perturb 2D Vertex by \ac{CTPL} Noise}
\label{alg:samplelevy2d}
\KwIn{Coordinates $x, y \in \mathbb{R}^2$, scale $c$, limit $r_{\max}$, tolerance $\delta$}
\KwOut{Perturbed coordinates $x', y'$}
Tune $\lambda$ using Algorithm~\ref{alg:autotune-lambda} so that $P[r \sim \mathrm{CTPL}(c, \lambda) \geq r_{\max}] < \delta$;\\
Draw $r \sim \mathrm{CTPL}(c, \lambda)$;\\
Draw $\theta \sim \mathrm{Uniform}[0, 2\pi)$;\\
$x' \leftarrow x + r \cos \theta$;\\
$y' \leftarrow y + r \sin \theta$;\\
return $(x', y')$;
\end{algorithm2e}

\begin{algorithm2e}[H]
    \caption{Motif-Extractor: Greedy selection of disjoint fragile motifs. Motifs are sorted by Table~\ref{tbl-degree-summary}: process largest $C_d$}
    \label{alg:motif_extractor}

    \KwIn{$G=(V,E)$ (graph), $D$ (degree vector), $B$ (binary adjacency matrix), $\text{degree\_priority} = \{6, 5, 4, 3, 2, 1\}$}
    \KwOut{$\mathcal{M}$ (set of disjoint motif hub vertices)}

    \BlankLine

    $\mathcal{M} \gets \emptyset$\;
    $\text{Covered} \gets \emptyset$\;

    \ForEach{$d$ in $\text{degree\_priority}$}{
        $\text{candidates} \gets \{v \in V \mid D[v] = d\}$\;
        \ForEach{$v$ in $\text{candidates}$}{
            \If{$v \notin \text{Covered}$}{
                $M_{\text{cand}} \gets \{v\} \cup \{u \in V \mid B[v, u] = 1\}$\;
                \If{$M_{\text{cand}} \cap \text{Covered} = \emptyset$}{
                    $\mathcal{M} \gets \mathcal{M} \cup \{v\}$\;
                    $\text{Covered} \gets \text{Covered} \cup M_{\text{cand}}$\;
                }
            }
        }
    }
    \KwRet{$\mathcal{M}$}\;
\end{algorithm2e}

\begin{algorithm2e}[H]
\caption{$\varepsilon$-Thickness Repair}
\label{alg:thick-repair}
\KwIn{Graph $G = (V, E)$ embedded in $\mathbb{R}^d$, separation threshold $\varepsilon > 0$}
\KwOut{Repaired graph $\widetilde G = (V \cup V_s, \widetilde E)$ satisfying $\varepsilon$-thickness}

$Q \gets E$;\\

\While{$Q \neq \emptyset$}{
    $e = (u, w) \gets Q.\textsc{Dequeue}()$\;

    $V_{\text{close}} \gets \{v \in V \setminus \{u,w\} : d(v, \text{segment}(e)) < \varepsilon\}$\;

    \eIf{$V_{\text{close}} \neq \emptyset$}{
        Choose $v \in V_{\text{close}}$ minimizing $d(v, \text{segment}(e))$\;

        $E \gets E \setminus \{(u,w)\}$\;
        $E \gets E \cup \{(u,v), (v,w)\}$\;

        $Q.\textsc{Enqueue}((u,v))$\;
        $Q.\textsc{Enqueue}((v,w))$\;
    }{
        \tcp{Edge $e$ satisfies thickness constraint, keep it}
    }
}

\Return{$\widetilde G = (V \cup V_s, \widetilde E)$}
\end{algorithm2e}

\subsection{Proof that optimal motif tiling is
NP-hard}\label{sec-np-hard-proof}

We formalise the optimal motif tiling problem and relate it to the
maximum weight independent set (MWIS) problem on a derived conflict
graph.

Let \(G = (V,E)\) be a fixed graph and let
\(\mathcal{M} = \{M_1,\dots,M_m\}\) be a finite family of candidate
motifs, where each motif is a vertex subset \(M_i \subseteq V\). A
feasible motif tiling is a subcollection
\(\mathcal{T} \subseteq \mathcal{M}\) such that the motifs are
vertex-disjoint, i.e., \(M_i \cap M_j = \varnothing\) for all
\(M_i, M_j \in \mathcal{T}\), \(i \ne j\). We define the coverage of a
tiling as the total number of vertices covered,
\(\mathrm{cov}(\mathcal{T}) = |\bigcup_{M_i \in \mathcal{T}} M_i| = \sum_{M_i \in \mathcal{T}} |M_i|\),
where the last equality uses disjointness. The optimal motif tiling
problem is to maximise \(\mathrm{cov}(\mathcal{T})\) subject to
\(\mathcal{T} \subseteq \mathcal{M}\) and
\(M_i \cap M_j = \varnothing\).

From \(\mathcal{M}\) we build a conflict graph \(H = (V_H, E_H)\) as
follows: for each motif \(M_i \in \mathcal{M}\), introduce a vertex
\(v_i \in V_H\); connect \(v_i\) and \(v_j\) by an edge in \(H\)
whenever the corresponding motifs overlap, i.e.,
\(\{v_i, v_j\} \in E_H \iff M_i \cap M_j \ne \varnothing\); and assign
each vertex \(v_i\) a non-negative weight \(w_i = |M_i|\).

An independent set \(I \subseteq V_H\) is a set of vertices with no
internal edges. By construction, \(I\) is independent if and only if
\(\mathcal{T}_I = \{M_i : v_i \in I\}\) is a feasible tiling. The weight
of an independent set \(I\) is
\(w(I) = \sum_{v_i \in I} w_i = \sum_{M_i \in \mathcal{T}_I} |M_i| = \mathrm{cov}(\mathcal{T}_I)\).
Thus the optimal motif tiling problem is equivalent to the maximum
weight independent set problem on \(H\):
\(\max_{\mathcal{T}\ \text{feasible}} \mathrm{cov}(\mathcal{T}) \Longleftrightarrow \max_{I\ \text{independent in } H} w(I)\).

The maximum weight independent set problem is NP-hard on general graphs.
Because any instance of MWIS can be represented as the conflict graph
\(H\) of a suitable motif family (via the construction above), solving
the optimal motif tiling problem is at least as hard as solving MWIS.
Consequently, the optimal motif tiling problem is NP-hard in general.

\textbf{AI usage disclaimer.} During the preparation of this manuscript,
we used generative AI tools to assist with implementation tasks (e.g.,
boilerplate code generation and plotting scripts), routine code
screening, and publication logistics (Quarto setup and formatting). All
substantive intellectual contributions, including problem formulation,
theory development, proofs, interpretation of results, and the final
textual content, are the result of human scholarly effort. The authors
take full responsibility for the accuracy and integrity of the final
manuscript.

\end{document}

%% file: Abbreviations.tex
\DeclareAcronym{SC}{
    short = {SC},
    long = {stochastic co-spectrality},
    tag = {abbrev},
}
\DeclareAcronym{S3I}{
    short = {S3I},
    long = {stochastic spectral separation index},
    tag = {abbrev},
}
\DeclareAcronym{TT}{
    short = {TT},
    long = {tempered tail},
    tag = {abbrev},
}
\DeclareAcronym{CTPL}{
    short = {CTPL},
    long = {calibrated tempered power-law},
    tag = {abbrev},
}